\documentclass[reqno,11pt]{amsart}
\calclayout
\usepackage{amsmath}
\usepackage{amsthm}
\usepackage{amssymb}
\usepackage{amsbsy}
\usepackage{amsfonts}
\usepackage{amstext}
\usepackage{amscd}
\usepackage{tikz}
\usepackage{graphicx}
\usepackage{verbatim}
\usepackage{bm}
\usepackage{commath,mathtools,setspace}
\usepackage{paralist}
\usepackage{extarrows}
\usepackage{color}

\usepackage[T1]{fontenc}
\usepackage{graphicx}
\usepackage{blindtext} % Package to generate dummy text throughout this template 
\usepackage{geometry}
\usepackage[colorlinks=true,linkcolor=blue,citecolor=blue]{hyperref}
\usepackage[colorinlistoftodos,prependcaption,textsize=small]{todonotes}

\definecolor{red}{RGB}{255,0,0}
\definecolor{green}{RGB}{0,100,0}
\definecolor{blue}{RGB}{0,0,255}

\newtheorem{theorem}{Theorem}[section]
\newtheorem{lemma}[theorem]{Lemma}

\newtheorem{corollary}[theorem]{Corollary}
\newtheorem{proposition}[theorem]{Proposition}
\newtheorem{notation}[theorem]{Notation}

\theoremstyle{remark}
\newtheorem{remark}[theorem]{Remark}
\newtheorem{definition}[theorem]{Definition}
\newtheorem{example}[theorem]{Example}

\newcommand{\mesh}{\mathop{\rm mesh}}
\newcommand{\lmesh}{\mathop{\rm lmesh}}
\renewcommand{\P}{\mathbb{P}}

\newcommand{\R}{\mathbb{R}}
\newcommand{\C}{\mathbb{C}}
\newcommand{\Z}{\mathbb{Z}}

\newcommand{\N}{\mathbb{N}}

\newcommand{\monicpols}{\mathbb P_n^*}
\newcommand{\monicpolreal}{\mathbb P_n^*(\mathbb R)}

\newcommand{\lag}{\widehat {L}}
\newcommand{\jac}{\widehat {J}}
\newcommand{\bes}{\widehat {B}}

\newcommand{\cc}{\mathbb{C}}

\newcommand{\nn}{\mathbb{N}}

\newcommand{\pp}{\mathbb{P}}

\newcommand{\rr}{\mathbb{R}}

\newcommand{\zz}{\mathbb{Z}}

\newcommand{\FF}{\mathcal{F}}

\newcommand{\raising}[2]{\left(#1\right)^{\overline{#2}}}
\newcommand{\falling}[2]{\left(#1\right)^{\underline{#2}}}

\newmuskip\pFqmuskip

\newcommand*\pFqN[6][8]{%
  \begingroup % only local assignments
  \pFqmuskip=#1mu\relax
  \mathcode`\,=\string"8000
  \begingroup\lccode`\~=`\,
  \lowercase{\endgroup\let~}\pFqcomma
  {}_{#2}F_{#3}{\left(\genfrac..{0pt}{}{#4}{#5};#6\right)}%
  \endgroup
}
\newcommand{\pFqcomma}{\mskip\pFqmuskip}

\newcommand*\pFq[5]{%
 {\ }_{#1}{\mathcal F}_{#2}{\left(\genfrac..{0pt}{}{#3}{#4};#5\right)}%
}

\newcommand*\HGF[5]{%
 {\ }_{#1} F_{#2}{\left(\genfrac..{0pt}{}{#3}{#4};#5\right)}%
}

\newcommand*\HGP[3]{%
 \mathcal{H}_{#1}{\left[\genfrac..{0pt}{1}{#2}{#3}\right]}%
}
\title[Finite Free Convolution of Hypergeometric polynomials]{Real roots of hypergeometric polynomials via finite free convolution} 
\author[A. Mart\'{\i}nez-Finkelshtein]{Andrei Mart\'{\i}nez-Finkelshtein}

\address[AMF]{Department of Mathematics, Baylor University, TX, USA, and Department of Mathematics, University of Almer\'{\i}a, Spain}

\email{A\_Martinez-Finkelshtein@baylor.edu}

\author[R.~Morales]{Rafael Morales}

\address[RM]{Department of Mathematics, Baylor University, TX, USA}

\email{rafael\_morales2@baylor.edu}

\author[D.~Perales]{Daniel Perales}

\address[DP]{Department of Mathematics, Texas A\&M University, TX, USA}

\email{daniel.perales@tamu.edu}

\date{\today}

\keywords{Hypergeometric polynomials; Finite free convolution; Free probability; Zeros}

\subjclass[2020]{Primary:  33C20; Secondary: 33C45, 42C05, 46L54}

\begin{document}

\begin{abstract}
We examine two binary operations on the set of algebraic polynomials, known as  multiplicative and additive finite free convolutions,  specifically in the context of hypergeometric polynomials. We show that the representation of a hypergeometric polynomial as a finite free convolution of more elementary blocks, combined with the preservation of the real zeros and interlacing by the free convolutions, is an effective tool that allows us to analyze when all roots of a specific hypergeometric polynomial are real. Moreover, the known limit behavior of finite free convolutions allows us to write the asymptotic zero distribution of some hypergeometric polynomials as free convolutions of Marchenko-Pastur, reciprocal Marchenko-Pastur, and free beta laws, which has an independent interest within free probability.
\end{abstract}

\maketitle

\tableofcontents

\section{Introduction}

The definition of the general hypergeometric function ${ }_{i+1} F_j$ with $i+1$ numerator and $j$ denominator parameters is well known, see 
\eqref{hypergeoSeries} below. If one of the numerator parameters is equal to a negative integer, say $ -n$, with $n \in \mathbb{N}$, then the series terminates and is a polynomial of degree $n$. The natural question that arises in connection with any polynomial is the location and behavior of its zeros, in particular, when they are all real (``real-rootedness''). If all its zeros are real, we also want to know additional properties like positivity/negativity, interlacing, and monotonicity with respect to the parameters. This has importance, among other matters, in the study of the Laguerre-Pólya class $\mathcal L$-$\mathcal P$ of entire functions (functions that can be obtained as a limit, uniformly on compact subsets of $\mathbb{C}$, of a sequence of negative-real-rooted polynomials, see \cite{MR4442978}).

The connection between ${ }_{i+1} F_j$ hypergeometric polynomials and some classical families of polynomials, in many cases orthogonal, yields straightforward answers to these questions, at least for small values of $i$ and $j$. But when $i\ge 1$ and $j\ge 2$, the problem becomes more difficult due to the limited number of tools that allow us to investigate the zero location.  

One of such tools is the idea of transformations acting on the space of polynomials. Several such transformations have ``zero-mapping'' properties,  the differentiation acting on polynomials with all real roots being the simplest example. Further examples of such linear transformations can be constructed within the theory of multiplier sequences, originated in \cite{polya1914uber}, see also \cite{MR1113932,MR0568321,MR1483603,MR3105315,MR3207673,MR2218995}. In the classical theory, multiplier sequences that preserve real zeros are characterized by means of certain analytic properties of their generating functions (e.g., that they belong to the $\mathcal L$-$\mathcal P$ class).  

Several of these transformations can also be written as a ``convolution'' of a given polynomial with another polynomial or function. Again, many results can be traced back to the work of Szeg\H{o}, Schur, Walsh, and others. Recently, several such transformations have been rediscovered as a finite analogue of free probability, named generically as finite free convolution of polynomials \cite{MR4408504}. They have a number of very useful properties, not only preserving real-rootedness, but also interlacing, monotonicity and even asymptotic distribution of zeros under certain conditions.

The connection between these polynomial convolutions and free probability is revealed in the asymptotic regime, when we consider the zero-counting measure (also known in this context as the empirical root distribution) of a polynomial of degree $n$ and let the degree tend to $\infty$ to obtain a limiting measure. Then the finite free convolution of polynomials turns into a free convolution of measures. This interesting connection has benefited both areas of research. On the one hand, the several relations between measures studied in free probability can guide our intuition on the type of relation that their polynomial analogues could satisfy, as well as provide a simple way to compute limiting measures using free probability. On the other hand, some properties that are clear in the context of discrete measures (such as zero-counting measures of polynomials) give a concrete explanation to phenomena that are not apparent when working with absolutely continuous measures.

In this paper, we examine two of such finite free convolutions, namely the multiplicative $\boxtimes_n$ (also known as Schur--Szeg\H{o} composition) and the additive $\boxplus_n$ convolutions, specifically in the context of hypergeometric polynomials. The main finding is that these operations have natural realizations in the class of these polynomials, providing an additional tool for studying their zeros. 
To make it more precise, as well as to provide a guide to facilitate the reader to navigate the unavoidable abundance of formulas and identities, we give a brief outline of the main highlights of this paper next.

We introduce all the necessary notation and facts in Section~\ref{sec:preliminaries}. In particular, given two complex  polynomials
\[
p(x)=\sum_{i=0}^n x^{n-i}(-1)^i e_i(p) \qquad \text{and} \qquad  q(x)=\sum_{i=0}^n x^{n-i}(-1)^i   e_i(q)
\]
of degree  $n$, the finite free additive convolution, $p\boxplus_n q$, and the finite free multiplicative  convolution, $p\boxtimes_n q$, are defined as:
$$ [p\boxplus_n q](x):=\sum_{k=0}^n x^{n-k}(-1)^k \sum_{i+j=k}\frac{(n-i)!(n-j)!}{n!(n-k)!}\, e_i(p) e_j(q), $$
and
$$
[p\boxtimes_n q](x): =\sum_{k=0}^n x^{n-k}(-1)^k \binom{n}{k}^{-1}  e_i(p) e_k(q)  .
$$

These operations are closed on the set of polynomials with all real positive roots, making them a useful tool to study real-rooted polynomials, their root interlacing, and root separation; see Subsections \ref{sec:freeconvol}  and \ref{sec:realrootednessfreeconv} below for details.

Our goal is to study the effect of these operations on the roots of hypergeometric polynomials
$$
\pFq{i+1}{j}{-n, \bm a}{\bm b }{x} :=\raising{\bm b}{n} \HGF{i+1}{j}{-n , \bm a}{\bm b }{x}= \raising{\bm b}{n} \, \sum_{k=0}^n \frac{\raising{-n}{k} \raising{\bm a}{k}}{\raising{\bm b}{k}} \frac{x^k}{k !}, 
$$
where $\bm a =(a_1, \dots, a_i)\in \R^i$ and $\bm b =(b_1, \dots, b_j)\in \R^j$ are vectors of parameters, and $\raising{\bm a}{k}:=\raising{a_1}{k}\raising{a_2}{k} \dots \raising{a_s}{k}$, with $\raising{a}{k}:=a(a+1)\dots(a+k-1)$, denotes the rising factorial; see Subsection \ref{subsec:polynomials}.

Section \ref{sec:convolutionHGP} contains the main results on the representation of finite free multiplicative and additive convolutions of two hypergeometric polynomials. For instance, we show that the free multiplicative convolution satisfies
$$
\pFq{i_1+1}{j_1}{-n, \bm a_1}{\bm b_1}{x} \boxtimes_n \pFq{i_2+1}{j_2}{-n, \bm a_2}{\bm b_2}{x}=\pFq{i_1+i_2+1}{j_1+j_2}{-n, \bm a_1, \bm a_2}{\bm b_1,\bm b_2}{x},
$$
(Theorem \ref{thm:multiplicativeConvHG}), while closed expressions for the additive convolution 
$$\pFq{i+1}{j}{-n,\ \bm a}{\bm b }{ x}\boxplus_n \pFq{s +1}{t}{-n,\ \bm c}{\bm d}{ x},$$
follow from factorizations (summation formulas) of hypergeometric functions of the form
$$
    \HGF{j_1}{i_1}{\bm a_1}{\bm b_1}{x}  \HGF{j_2}{i_2}{\bm a_2}{\bm b_2}{ x}=\HGF{j_3}{i_3}{\bm a_3}{\bm b_3}{x},
$$
(Theorem~\ref{thm.additive.conv.HG} and Corollary \ref{cor.additive.conv.as.poruduct.HG}). 
These formulas allow us to assemble more complicated hypergeometric polynomials from simpler hypergeometric ``building blocks''. 

Combining knowledge of the behavior of the zeros of these blocks with the zero-preserving properties of the finite free convolution, we can obtain further results on zeros of hypergeometric polynomials (Section~\ref{sec:compilation}). For small values of $i$ and $j$, the ${}_{i+1}\mathcal{F}_{j}$ hypergeometric polynomials correspond to classical families of polynomials: Laguerre, Bessel, and Jacobi. Their root location has been extensively studied, with very precise descriptions on when the polynomials are real-rooted, when they interlace, and several results on the asymptotic distribution of the zero counting measures.
A combination of this knowledge with the results of Section \ref{sec:convolutionHGP} yields a systematic approach to the construction of families of real-rooted hypergeometric polynomials for larger $i$ and $j$. 

For instance, here are some of the general facts we can establish about zeros of the hypergeometric polynomial 
$$
p(x)= \pFq{i+1}{j}{-n,\ \bm a}{ \bm b}{x}.
$$
\begin{itemize}
    \item If $b_1,\dots, b_j>0$ and  $a_1,\dots,a_i<-n+1$ then $p$ is real-rooted with all the roots of the same sign (positive if $i$ is even, or negative if $i$ is odd), see Theorem \ref{prop:42bis}.

    \item If $j\geq i$, $b_1,\dots, b_j > 0$ and $a_1,\dots, a_i\in \rr$ are such that $a_s\geq n-1+b_s$ for $s=1,\dots,i$, then $p$ has all positive roots, see Theorem \ref{prop:42}.
\end{itemize}

For the ${}_{2}\mathcal{F}_{2}$, ${}_{1}\mathcal{F}_{3}$, and ${}_{3}\mathcal{F}_{2}$ polynomials, we provide more specific results in Sections \ref{sec:2F2real} and \ref{sec:3F2real}. For the reader's convenience, we have compiled the main combinations of the hypergeometric parameters for which the corresponding polynomials are real-rooted in Tables \ref{tab:2F2}--\ref{tab:3F2case3extension}. However, we want to make clear that neither these results are exhaustive nor we consider the free convolution to be the universal tool for establishing the real-rootedness of such polynomials. Instead, our goal is to illustrate how this approach can yield some new and non-trivial results or provide alternative proofs of known facts. 

Finally, in Section \ref{sec:free.prob.asymototics} we use a simple reparametrization to recast the previous results in the framework of finite free probability in the hope of giving additional intuition or insight to the readers more familiar with this field.  Moreover, with this new reformulation the asymptotic root-counting measure of hypergeometric polynomials is reduced to studying the distribution of the addition and multiplication of free random variables that obey the Marchenko-Pastur, reciprocal Marchenko-Pastur, or free beta laws.

This text is part of a larger project that includes an article \cite{MFMP} in preparation.

\section{Preliminaries }
\label{sec:preliminaries}

\subsection{Polynomials and their coefficients} \label{subsec:polynomials}\

We start by introducing some notation. In what follows, $\P_n$ stands for all algebraic polynomials of degree $\le n$, and $\monicpols \subset \P_n$ is the subset of monic polynomials of degree $n$. Also, for $K\subset \C$, we denote by $\P_n(K) $ (resp., $\monicpols(K)$) the subset of polynomials of degree $\le n$ (resp., monic polynomials of degree $n$) with all zeros in $K$. In particular, $\monicpolreal$ denotes the family of real-rooted monic polynomials of degree $n$, $\monicpols(\R_{\ge 0})$ is the subset of $\monicpolreal$ of polynomials having only non-negative roots, etc. 

Every polynomial $p$ of degree $n$ can be written in the form
\begin{equation}\label{monicP}
    p(x)=\sum_{j=0}^n x^{n-j} (-1)^j e_j(p).
\end{equation}
Since we do not require a priori $e_0(p)\neq 0$, notation $e_j(p)$ implicitly assumes the dependence on $n$. It is convenient to keep it in mind, although we will avoid mentioning it explicitly to simplify the notation.

If $p$ is a polynomial of degree $n$, then  $e_j(p)/ e_0(p)$ are just the symmetric sums of its roots: denoting by  $\lambda_1(p),\dots, \lambda_n(p)$ the roots of $p$ (in the case when $p$ is real-rooted, we use the convention that $\lambda_1(p)\geq \lambda_2(p)\geq \dots \geq \lambda_n(p)$), then 
\begin{equation}
    \label{zerossymmetric}
e_j(p)/ e_0(p)=\sum_{sym} \lambda_1(p)\lambda_2(p)\dots \lambda_j(p):=\sum_{1\leq i_1<i_2<\dots< i_j\leq n} \lambda_{i_1}(p)\lambda_{i_2}(p)\dots \lambda_{i_j}(p).
\end{equation}

One simple observation that we will use later is that if $p$ is \eqref{monicP}, and that for a constant $c\in \C $,  
$$
 q(x):=x^n p(c/x) 
$$
then the coefficients for $q$ are
\begin{equation}
    \label{eq:monicreversed}
e_j(q) = (-1)^n c^j e_{n-j}(p), \quad j=0, 1, \dots, n.     
\end{equation}

\subsection{Hypergeometric polynomials} \label{subsec:HGpolynomials}\

Rising and falling factorials play a crucial role in our calculations. The \textbf{rising factorial} (also, \textbf{Pochhammer's symbol\footnote{\, Another standard notation for the raising factorial is $(a)_j$. We prefer to use the notation defined here.}})  for $a\ne 0$ and $j\in \mathbb Z_{\ge 0}:= \N \cup \{0\}$ is  
$$
\raising{a}{j}:=a(a+1)\dots(a+j-1)= \frac{\Gamma(a+j)}{\Gamma(a)}, \quad \raising{a}{0}:=1,$$
while the \textbf{falling factorial} is defined as
$$
\falling{a}{j}:=a(a-1)\dots(a-j+1)= \raising{a-j+1}{j}, \quad \falling{a}{0}:=1.
$$
Notice the obvious useful relations
\begin{equation}
    \label{connectionRandL}
    \falling{a}{j}=(-1)^j \raising{-a}{j}, 
\end{equation}
as well as 
\begin{equation}
    \label{connectionRandL2}
     \raising{a}{n} = \raising{a}{n-j} \falling{a+n-1}{j}, \quad 0\le j \le n.
\end{equation}

A generalized \textbf{hypergeometric series} \cite{MR2656096, MR2723248} is an expression 
\begin{equation} \label{hypergeoSeries}
\HGF{i+1}{j}{a_0, a_1, \dots, a_i}{b_1, \dots, b_j}{x}=\sum_{k=0}^\infty  \frac{\raising{a_0}{k} \raising{a_1}{k}\dots\raising{a_i}{k}}{\raising{b_1}{k} \dots\raising{b_j}{k}} \frac{x^k}{k!}.
\end{equation}
If $\bm a =(a_1, \dots, a_i)\in \R^i$ is a vector (tuple), we understand by
$$
\raising{\bm a}{k} =\prod_{s=1}^i \raising{a_s}{k},
$$
and therefore, with $\bm a =(a_1, \dots, a_i)\in \R^i$ and $\bm b =(b_1, \dots, b_j)\in \R^j$, we can write 
\begin{equation} \label{hypergeoSeriesBis}
\HGF{i+1}{j}{a_0, \bm a}{\bm b}{x}=\sum_{k=0}^\infty  \frac{\raising{a_0}{k} \raising{\bm a }{k} }{\raising{\bm b }{k}  } \frac{x^k}{k !}.
\end{equation}
In the particular case when $a_0$ is a negative integer, the series is terminating and defines a polynomial. More precisely, for $n\in \N$, 
$$
 \HGF{i+1}{j}{-n, \bm a  }{\bm b }{x}=\sum_{k=0}^n  \frac{\raising{-n}{k} \raising{\bm a }{k} }{\raising{\bm b }{k}  } \frac{x^k}{k !}
$$
is a (generalized) \textbf{hypergeometric polynomial} of degree $ \le n$, as long as 
\begin{equation*}  
    b_1,\dots b_j \in \C \setminus \{0, -1,-2,\dots, -n+1 , -n\}.
\end{equation*}

In what follows, it will be more convenient for us to work with the \textbf{normalized terminating hypergeometric series}
\begin{equation} \label{hypergeoSeriesNormalized}
\pFq{i+1}{j}{-n, a_1, \dots, a_i}{b_1, \dots, b_j}{x}:= \left( \prod_{k=1}^j \raising{b_k}{n}\right) \HGF{i+1}{j}{-n, a_1, \dots, a_i}{b_1, \dots, b_j}{x}.
\end{equation}
Correspondingly, we say that the  (generalized) \textbf{hypergeometric polynomial} $p$ of degree $n\in \mathbb Z_{\ge 0} $ is in \textbf{standard normalization} if\,\footnote{\, If $j=0$, then the factor $\raising{\bm b}{n}$ in the right-hand side of \eqref{hypergeoPolyn} equals $1$, and thus we identify $_{i+1}F_0$ and $_{i+1}\mathcal F_0$.}
\begin{equation} \label{hypergeoPolyn}
p(x)=\pFq{i+1}{j}{-n, \bm a}{\bm b }{x} = \raising{\bm b}{n} \, \sum_{k=0}^n \frac{\raising{-n}{k} \raising{\bm a}{k}}{\raising{\bm b}{k}} \frac{x^k}{k !}.
\end{equation}
Notice that with this normalization, $p$   is a polynomial in both $x$ and its parameters $a_s$, $b_s$: using \eqref{connectionRandL2}, we can rewrite the expression \eqref{hypergeoPolyn} as 
\begin{equation}
    \label{consequenceconnectionRandL2}
      p(x)=\pFq{i+1}{j}{-n, \bm a}{\bm b }{x} =   \sum_{k=0}^n \ \raising{-n}{k} \raising{\bm a}{k} \raising{\bm b+k}{n-k} \frac{x^k}{k !};
\end{equation}
so that, for this polynomial, written in the form \eqref{monicP}, we have that 
\begin{equation} \label{leadingcoeffHGFBis}
    e_k(p)=(-1)^{n+k(i+j)}  \raising{\bm a}{n} \binom{n}{k} \frac{\raising{-\bm b-n+1}{k}}{\raising{-\bm a-n+1}{k}} =(-1)^{n }   \binom{n}{k}
    \raising{\bm a}{n-k} \raising{\bm b+n-k}{k},   
\end{equation}
and in particular,
\begin{equation} \label{leadingcoeffHGF}
    e_0(p)=(-1)^n  \raising{\bm a}{n} , \quad e_n(p)=(-1)^n  \raising{\bm b}{n}.
\end{equation}
These expressions show that the polynomial is of degree exactly  $n$ if and only if
\begin{equation}\label{assumptionsHGFNew}
  a_1,\dots, a_i   \in \C \setminus \{0, -1,-2,\dots, -n+1  \},
\end{equation}
constraint that we will consider enforced henceforth. For the rest of the values of the parameters, we always understand by the hypergeometric polynomial in standard normalization the expression \eqref{consequenceconnectionRandL2}. 

Since the following discrete set will appear very frequently in this work, we will introduce the notation
\begin{equation}
    \label{notationZn}
    \Z_n := \{0,  1, 2,\dots,  n-1  \},
\end{equation}
understanding by $-\Z_n$ the set $\{0, -1,-2,\dots, -n+1  \}$. In particular, condition \eqref{assumptionsHGFNew} can be written as $a_1,\dots, a_i \notin \left(- \Z_n\right)$. 

Notice that as in \eqref{eq:monicreversed}, for a constant $c\in \C$, the polynomial
$$q(x):=x^n p(c/x)= x^n {}_{i+1} F_j\left(\begin{matrix} -n, \bm a \\ \bm b
\end{matrix};\frac{c}{x}\right),$$
written in the form \eqref{monicP}, has coefficients  
$$
e_k(q)=c^k \binom{n}{k} \frac{ \raising{\bm a}{k} }{\raising{\bm b}{k} }.
$$ 

Comparing it with \eqref{eq:monicreversed} and \eqref{leadingcoeffHGFBis}, direct computations lead to the following identity:
\begin{lemma}
\label{lem:pFq.reversed}
The following identity holds true for $x\neq 0$:
$$
\pFq{i+1}{j}{-n,\ \bm a}{ \bm b}{x}=
\left((-1)^{i+1} x\right)^n \pFq{j+1}{i}{-n,\ -n -{\bm b} +1}{ -n -{\bm a} +1}{(-1)^{i+j}\frac{1}{x}}.
$$
\end{lemma}

\subsection{Some classical hypergeometric polynomials}\label{sec:classicalHGP}\

Many classical families of polynomials are hypergeometric. In this section, we summarize some basic information needed in what follows. Further details and formulas can be found in \cite{MR2542683, MR2656096, MR2723248, szego1975orthogonal}.

We will use, whenever possible, the standard normalization introduced in \eqref{hypergeoPolyn}.

Since for  $a\in \C$, and for each $n\in \N$, 
$$ 
p_n(x)=(x-a)^n= \sum_{k=1}^n x^{n-k} (-1)^k \binom{n}{k} a^k,
$$ 
we can conclude that for $a\neq 0$,  
\begin{equation}
    \label{Power1F0}
    (x-a)^n= (-a)^n \ {}_1 F_0\left(\begin{matrix} -n \\ \cdot 
\end{matrix}; \frac{x}{a} \right).
\end{equation}

The   (generalized) \textbf{Laguerre polynomials} of degree $n$ and parameter $\alpha\in \C$, in their traditional normalization, are defined as
\begin{equation}
    \label{laguerreP}
    L^{(\alpha)}_{n}(x):=\sum_{k=0}^{n}
\frac{ \falling{n+\alpha}{n-k}}{k!(n-k)!}\, (-x)^{k}.
\end{equation}
Additionally, for $k=1,2,\dots,n$,
\begin{equation} \label{reductionLaguerre}
    L_n^{(-k)}(x)=(-x)^k \, \frac{(n-k) !}{n !}\,  L_{n-k}^{(k)}(x).
\end{equation}
When $\alpha\geq -1$,   $ L^{(\alpha)}_{n} $ are orthogonal on $[0, \infty)$, so that all their roots are simple and $L^{(\alpha)}_{n} \in \P_n(\R_{>0})$. By \eqref{reductionLaguerre}, for $\alpha=-1,-2,\dots,-n$, $L^{(\alpha)}_{n}\in \P_n(\R_{\ge 0})$, with a unique multiple root of order $-\alpha$ at $ 0$, and all other roots distinct and positive. Moreover,  $L^{(\alpha)}_{n}\in \P_n(\R )$   even when  $\alpha\in(-2,-1)$, with   $n-1$ positive zeros and one negative zero, see, e.g.,~\cite[\S 6,73]{szego1975orthogonal}.  
 
As a hypergeometric function one has
\begin{equation} \label{HGFLaguerre}
L^{(\alpha)}_{n}(x)= 
%\frac{\raising{\alpha+1}{n}}{n!} \ {}_1 F_1\left(\begin{matrix} -n\\ \alpha + 1 
%\end{matrix};x\right) =
\frac{1}{n!}\, 
\pFq{ 1}{1}{-n}{\alpha +1}{x}.
\end{equation}

By Lemma~\ref{lem:pFq.reversed}, the reciprocal polynomials are
\begin{equation}
    \label{reversedLaguerre}
    q(x)=x^n L^{(\alpha)}_{n}(-1/x)=\frac{1}{n!}\, \pFq{ 2}{0}{-n, -n-\alpha}{\cdot}{x},
\end{equation}
and are known as \textbf{Bessel polynomials}.

Finally, the \textbf{Jacobi polynomials} of degree $n$ and parameters $\alpha,\beta\in \C$ are
\begin{align*}
P^{(\alpha,\beta)}_{n}(x)&:=\frac{1}{n!}\sum_{k=0}^{n} \binom{n}{k} \falling{n+\alpha+\beta+k}{k} \falling{\alpha + n}{n-k} \left(\frac{x-1}{2}\right)^k \\
&=\frac{1}{n!}\sum_{k=0}^{n} \binom{n}{k} \raising{n+\alpha+\beta+1}{k} \raising{\alpha + k+1}{n-k} \left(\frac{x-1}{2}\right)^k.    
\end{align*}
In consequence,  
\begin{align} \label{JacobiPfla1}
P^{(\alpha,\beta)}_{n}(x)&:= \frac{1}{n!}\, \pFq{ 2}{1}{-n, n+ \alpha+ \beta+1}{\alpha +1}{\frac{1-x}{2}} 
\\ 
\label{JacobiPfla3}
& =\frac{(-1)^n}{n!}\, \pFq{ 2}{1}{-n, n+ \alpha+ \beta+1}{\beta +1}{\frac{1+x}{2}}
\\ 
\label{JacobiPfla2}
&=\frac{1}{n!}\, \left(\frac{1+x}{2}\right)^n \, \pFq{ 2}{1}{-n, -n-\beta}{\alpha + 1 }{\frac{x-1}{x+1}}.
\end{align}

The following  identities are well known:
\begin{equation}\label{transfJacobi}
\begin{split}
P_n^{(\alpha, \beta)}(-x)& =(-1)^n P_n^{(\beta, \alpha)}(x), \\
P_n^{(\alpha, \beta)}(x) & =\left(\frac{1-x}{2}\right)^n P_n^{(-2 n-\alpha-\beta-1, \beta)}\left(\frac{x+3}{x-1}\right), \\
P_n^{(\alpha, \beta)}(x)& =\left(\frac{1+x}{2}\right)^n P_n^{(\alpha,-2 n-\alpha-\beta-1)}\left(\frac{3-x}{x+1}\right),
\end{split}
\end{equation}
see \cite[\S 4.22]{szego1975orthogonal}.

The classical Jacobi polynomials (that correspond to parameters $\alpha, \beta>-1$)  are orthogonal on $[-1,1]$ with respect to the weight function $(1-x)^\alpha(1+x)^\beta$. Consequently, all their zeros are simple and belong to the interval $(-1,1)$. If $\alpha$ or $\beta$ are in $[-2,-1]$, we can also guarantee that $P_n^{(\alpha, \beta)}$ has real zeros, but not all of them in $[-1,1]$, see \cite{driver2018zeros}. 

Moreover, $P_n^{(\alpha, \beta)}(x)$ may have a multiple zero, but always at $x=\pm1$:
\begin{itemize}
    \item at $x=1$, if $\alpha \in\{-1, \ldots,-n\}$. More precisely, for $k \in\{1, \ldots, n\}$, we have (see \cite[Eq. (4.22.2)]{szego1975orthogonal}),
\begin{equation}
    \label{reductionJacobi1}
    P_n^{(-k, \beta)}(x)=\frac{\raising{n+\beta+1-k}{k} }{\raising{n-k+1}{k}} \left(\frac{x-1}{2}\right)^k P_{n-k}^{(k, \beta)}(x).
\end{equation}
This implies, in particular, that $P_n^{(-k, \beta)}(x) \equiv 0$ if additionally $\max \{k,-\beta\} \leq n \leq k-\beta-1$. 

    \item at $x=-1$ if $\beta \in\{-1, \ldots,-n\}$. More precisely, when $l \in\{1, \ldots, n\}$. 
    \begin{equation}
    \label{reductionJacobi1bis}
    P_n^{(\alpha, -l)}(x)=  \frac{\raising{n+\alpha+1-l}{l} }{\raising{n-l+1}{l}} \left(\frac{x+1}{2}\right)^l P_{n-l}^{(\alpha, l)}(x).
\end{equation}
    
The formulas above show that when both $k, l \in \mathbb{N}$ and $k+l \leq n$, we have
   \begin{equation}
   \label{reductionJacobi3new}
P_n^{(-k,-l)}(x)=2^{-k-l}(x-1)^k(x+1)^l P_{n-k-l}^{(k, l)}(x),
\end{equation}
with $P_n^{(-k,-l)}\equiv 0$ if $n\le k+l-1$. \\

    \item at $x=\infty$ (which means a degree reduction):  when $n+\alpha+\beta=-k \in\{-1, \ldots,-n\}$,
 \begin{equation*}
P_n^{(\alpha, \beta)}(x)=\frac{\Gamma(n+\alpha+1)}{\Gamma(k+\alpha)} \frac{(k-1) !}{n !} P_{k-1}^{(\alpha, \beta)}(x) ;
\end{equation*}
see \cite[Eq. (4.22.3)]{szego1975orthogonal}.  
\end{itemize}

These identities can be easily reformulated in terms of the hypergeometric polynomials in standard normalization using that
\begin{align} \label{HPJacobiPfla1}
\pFq{ 2}{1}{-n, a}{b }{x} 
&:=  n!  \, P^{(b-1,-n+a-b)}_{n}(1-2x) =(-1)^n  n! \, P^{(-n+a-b,b-1)}_{n}(2x-1) .
\end{align}
Furthermore, for $k\in \Z_n  $, we have the following consequences of applying formula  \eqref{JacobiPfla1} to \eqref{reductionJacobi1}--\eqref{reductionJacobi3new}:
\begin{itemize}
\item formula \eqref{reductionJacobi1} gives
\begin{equation}
    \label{2ndreduction2F1}
\pFq{2}{1}{-n, b + k+1}{-n+k+1}{x}
        = \raising{b+k+1}{n-k} \, (-x)^{n-k} \, \pFq{2}{1}{-k, b + n+1}{n-k+1}{x}   
\end{equation}
(in order to enforce condition \eqref{assumptionsHGFNew} in both sides, we assume that
\begin{equation}
    \label{b+k}
  b+k \notin (-\Z_n),  
\end{equation}
and evaluate the polynomials using the equivalent expression \eqref{consequenceconnectionRandL2});
 \item formula \eqref{reductionJacobi1bis} produces
    \begin{equation}
    \label{reduction2F1}
\pFq{2}{1}{-n, b + k}{b}{x}
        =  \raising{b+k}{n-k} \, (1-x)^{n-k} \, \pFq{2}{1}{-k, b + n}{b}{x}   
\end{equation}
(again, we require \eqref{b+k}); 

\item finally, by \eqref{reductionJacobi3new},  for $0\le k \le j \le n$,
  \begin{equation}
    \label{reduction2F14bis}
\pFq{2}{1}{-n, j-k+1}{1-k}{x}
        =   \raising{j-k+1}{n+k-j}\,(-x)^{k}\,(1-x)^{n-j}  \, \pFq{2}{1}{k-j, n+1}{k+1}{x}  . 
\end{equation}
\end{itemize}

\subsection{Finite free convolution of polynomials} \label{sec:freeconvol} \

In this section, we summarize some definitions and results on the finite free additive and multiplicative convolutions that will be used throughout the paper. These correspond to two classical polynomial convolutions studied a century ago by Szeg\H{o} \cite{szego1922} and Walsh \cite{Walsh1922} that were recently rediscovered in \cite{MR4408504} as expected characteristic polynomials of the sum and product of randomly rotated matrices.

\subsubsection{Multiplicative finite free convolution}
\begin{definition}[\cite{MR4408504}] \label{def:MultiplicativeConv}
Given two polynomials, $p$ and $q$, of degree at most $n$, the 
\textbf{$n$-th multiplicative finite free convolution} of $p$ and $q$, denoted as $p\boxtimes_n q$, is a polynomial of degree at most $n$, which can be defined in terms of the coefficients of polynomials written in the form \eqref{monicP}: if
\begin{equation}
    \label{defFMULTCONV}
    p(x)=\sum_{j=0}^n x^{n-j}(-1)^j e_j(p) \quad \text { and } q(x)=\sum_{j=0}^n x^{n-j}(-1)^j e_j(q),
\end{equation}
then
$$
[p\boxtimes_n q](x) = \sum_{k=0}^n x^{n-k}(-1)^k e_k(p\boxtimes_n q),
$$
with
\begin{equation}
    \label{coeffMultConv}
  e_k(p\boxtimes_n q) :=  \binom{n}{k}^{-1} e_j(p) e_j(q).
\end{equation}
\end{definition}
In particular, if $p, q \in \monicpols$, then also $p\boxtimes_n q \in \monicpols$. 
\begin{remark}
This operation, originally introduced in \cite{szego1922}, is known in the study of geometry of polynomials and of the Laguerre-Polya class as the \textbf{Schur--Szeg\H{o} composition}, see, e.g.~\cite{MR2242036, MR2375107}. It can also be regarded as the Hadamard product of $p$ and $q$, up to a sign and binomial factor; see~\cite{MR2825108}. We will use the term ``multiplicative finite free convolution'' for uniformity of terminology.
\end{remark}

The multiplicative finite free convolution is a linear operator  from $\P_n \times \P_n$ to $\P_n$: if $p,q,r\in \P_n$,  and $\alpha\in\rr$, then 
\begin{equation*}
    (\alpha p+q)\boxtimes_n r= \alpha (p\boxtimes_n r)+q\boxtimes_n r.
\end{equation*}

Definition \ref{def:MultiplicativeConv} allows us to establish easily that 
    \begin{equation}
        \label{identityMult1}
        p(x) \boxtimes_n (x-1)^n = p(x),
    \end{equation}
(that is, $(x-1)^n$ is an identity for the multiplicative convolution), as well as that
\begin{equation}
    \label{identityMult2}
    p(x) \boxtimes_n (x-\alpha)^n = \alpha^n p\left(\frac{x}{\alpha}\right), \quad \alpha\neq 0.
\end{equation}
This motivates the following definition:
\begin{definition} \label{def:multInverse}
Given $p  $ of degree $n$, the polynomial $q\in \mathbb P_n$ such that $p(x)\boxtimes_n q(x)=(x-1)^n$ is called the \textbf{inverse of $p$ under the multiplicative (finite free) convolution}.
\end{definition}
Notice that such an inverse does not always exist, since by \eqref{coeffMultConv}, a coefficient of $p\boxtimes_n q$ vanishes if the corresponding coefficient of $p$ or $q$ is $0$.

\subsubsection{Additive finite free convolution}

\begin{definition}[\cite{MR4408504}]
Given two polynomials, $p$ and $q$, of degree at most $n$, the \textbf{$n$-th additive finite free convolution} of $p$ and $q$, denoted as $p\boxplus_n q$, is a polynomial of degree at most $n$, defined in terms of the coefficients of polynomials written in the form \eqref{monicP}: if
\begin{equation}
    \label{defFADDCONV}
    p(x)=\sum_{j=0}^n x^{n-j}(-1)^j e_j(p) \quad \text { and } q(x)=\sum_{j=0}^n x^{n-j}(-1)^j e_j(q),
\end{equation}
then
$$
[p\boxplus_n q](x) = \sum_{k=0}^n x^{n-k}(-1)^k e_k(p\boxplus_n q),
$$
with
\begin{equation}
    \label{coeffAdditiveConv}
e_k(p\boxplus_n q) :=  \sum_{i+j=k} \frac{(n-i) !(n-j) !}{n !(n-k) !} \, e_i(p) e_j(q)
\end{equation}
(and thus, $e_0(p\boxplus_n q)=e_0(p)e_0(  q)$).
\end{definition}
 
Equivalently, the additive free convolution can be defined as
\begin{equation}
    \label{defAdditiveConv}
    [p\boxplus_n q](x):=\frac{1}{n!} \sum_{i=0}^n p^{(i)}(x) q^{(n-i)}(0)=\frac{1}{n!} \sum_{i=0}^n q^{(i)}(x) p^{(n-i)}(0).
\end{equation}
Especially useful for our purposes will be the third equivalent definition in terms of the associated differential operators. Namely, given the polynomial $p$ in \eqref{monicP}, define
a differential operator $D_p$ as
\begin{equation}\label{diffOperatorforP}
    D_p:=\sum_{j=0}^n (-1)^j \frac{e_j(p)}{\falling{n}{j}} \left(\frac{\partial}{\partial x}\right)^j.
\end{equation}
Then 
\begin{equation}\label{correspondencediffOperatorforP}
    D_p[x^n]= \sum_{j=0}^n x^{n-j} (-1)^j e_j(p) =p(x).
\end{equation}
Clearly, the correspondence $p \leftrightarrow D_p$ is a bijection between $\mathbb P_n$ and linear differential operators of degree $\le n$ with constant coefficients. For future reference, it is also convenient to observe that for a constant $c\neq 0$,
\begin{equation}\label{correspondencediffOperatorforP2}
   \left( \sum_{j=0}^n (-1)^j \frac{e_j(p)}{\falling{n}{j}} \left(c \frac{\partial}{\partial x}\right)^j \right) [x^n]= \sum_{j=0}^n x^{n-j} (-1)^j c^j e_j(p) = c^n p(x/c).
\end{equation}
  
Now, back to the additive free convolution: if $D_p$ and $D_q$ are the differential operators, corresponding to polynomials $p$ and $q$ in \eqref{defFADDCONV}, that is, if $p(x)=D_p [x^n]$ and $q(x)=D_q [x^n]$, then
\begin{equation} \label{thirdDefAdditiveConv}
    [p\boxplus_n q](x)=D_p[D_q[x^n]]=D_q[D_p[x^n]].
\end{equation} 
It follows from here (or directly from the definition) that
\begin{equation} \label{inversion}
    [p(-x)]\boxplus_n [q(-x)]=(-1)^n[p\boxplus_n q](-x).
\end{equation} 
 
When at least one of the polynomial is of degree strictly smaller than $n$, then (see Lemma 1.16 of \cite{MR4408504})
\begin{equation}
    \label{derivativeFreeConv}
    p\boxplus_n q = \frac{1}{n} p'\boxplus_{n-1} q,
\end{equation}
whenever $p\in \P_n$ and $q\in \P_{n-1}$. In particular, if $\deg p=n$ and $0\le k \le n$, then
$$
p \boxplus_n x^k = \frac{(n-k)!}{n!} \, p^{(n-k)},
$$
which also easily follows from \eqref{defAdditiveConv}.

The additive finite free convolution is a linear operator  from $\P_n \times \P_n$ to $\P_n$: if $p,q,r\in \P_n$,  and $\alpha\in\rr$, then 
\begin{equation*}
    (\alpha p+q)\boxplus_n r= \alpha (p\boxplus_n r)+q\boxplus_n r.
\end{equation*}
 
The three equivalent definitions of additive finite free convolution allow us to establish easily that 
    \begin{equation}
        \label{shift}
        p(x) \boxplus_n (x-\alpha)^n = p(x-\alpha), \quad p\in \P_n,
    \end{equation}
    so that, in particular, $p \boxplus_n x^n = p$. In other words, $x^n$ is an identity for the additive convolution. This motivates the following definition:
    \begin{definition} \label{def:additiveInverse}
  Given $p  $ of degree $n$, the polynomial $q\in \mathbb P_n$ such that $p(x)\boxplus_n q(x)=x^n$ is called the \textbf{inverse of $p$ under the additive (finite free) convolution}.
\end{definition}
Such an inverse always exists (see \cite[Corollary 6.2]{marcus}) and can be constructed recursively, using \eqref{coeffAdditiveConv}. 

Moreover, $p \boxplus_n q =0$ if and only if $\deg(p)+\deg(q) < n$, or if $\deg p=n$ then $q\equiv 0$ (last observation follows from \eqref{defAdditiveConv} and the fact that when $\deg p=n$, polynomials $p^{(i)}$, $i=0, 1, \dots, n$, form a basis of $\mathbb P_n$). This also shows that the inverse of any $p\in \P_n$ under the additive (finite free) convolution $\boxplus_n$ is unique.

\subsection{Real roots, interlacing, and free finite convolution} \label{sec:realrootednessfreeconv}\

A very important fact is that in many circumstances the finite free convolution of two polynomials with real roots also has all its roots real. Here, we use the notation introduced at the beginning of Section \ref{subsec:polynomials}. 

\begin{proposition}[Szeg\H{o} \cite{szego1922}, Walsh \cite{Walsh1922}]
    \label{prop:realrootedness}
    Let $p, q \in \P_n$. Then
    \begin{enumerate}
        \item[(i)]  $p,q\in \P_n(\rr) \ \Rightarrow \ p\boxplus_n q\in \P_n(\rr)$.
        \item[(ii)] $p\in \P_n(\rr),\ q\in \P_n(\rr_{\geq 0}) \ \Rightarrow \ p\boxtimes_n q\in \P(\rr)$.
        \item[(iii)] $p,q\in \P_n(\rr_{\geq 0}) \ \Rightarrow \ p\boxtimes_n q\in \P(\rr_{\geq 0})$.
    \end{enumerate}
\end{proposition}
Taking into account that $p(-x)=p\boxtimes_n(x+1)^n$ (see \eqref{identityMult2} with $\alpha=-1$), a simple consequence of this proposition is that additionally the following ``rule of signs'' applies:
\begin{itemize}
\item $p,q\in \P_n(\rr_{\leq 0}) \ \Rightarrow \ p\boxtimes_n q\in \P(\rr_{\geq 0})$

\item $p\in \P_n(\rr_{\leq 0}),\ q\in \P_n(\rr_{\geq 0}) \ \Rightarrow \ p\boxtimes_n q\in \P(\rr_{\leq 0})$.
\end{itemize}
\begin{remark}
Multiplicative finite free convolution can also be considered in the framework of finite multiplier sequences \cite{MR1483603, MR0568321, MR2242036}, where the zero preservation results are known.
\end{remark}

\begin{definition}[Interlacing] 
Let
$$
p(x)=e_0(p)\prod_{j=1}^n \left(x-\lambda_j(p)\right) \in \P_n(\R), \quad \lambda_1(p) \leq \dots \leq \lambda_n(p),
$$
and 
$$
q(x)=e_0(q)\prod_{j=1}^m \left(x-\lambda_j(q)\right) \in \P_m(\R), \quad \lambda_1(q) \leq \dots \leq \lambda_m(q).
$$
We say that $q$ \textbf{interlaces} $p$ (or, equivalently, that \textbf{zeros of $q$ interlace zeros of $p$}, see, e.g.,~\cite{MR3051099}), and denote it $p \preccurlyeq q$, if 
\begin{equation} \label{interlacing1}
    m=n \quad \text{and} \quad \lambda_1(p) \leq \lambda_1(q) \leq \lambda_2(p) \leq \lambda_2(q) \leq \cdots \leq  \lambda_n(p) \leq \lambda_n(q),
\end{equation}
or if
\begin{equation} \label{interlacing2}
    m=n-1 \quad \text{and} \quad \lambda_1(p) \leq \lambda_1(q) \leq \lambda_2(p) \leq \lambda_2(q) \leq \cdots \leq  \lambda_{n-1}(p) \leq \lambda_{n-1}(q) \le \lambda_n(p).
\end{equation}
Furthermore, we use the notation $p \prec q$ when all inequalities in \eqref{interlacing1} or \eqref{interlacing2} are strict.
\end{definition}

The Hermite-Kakeya (a.k.a.~the Hermite–Kakeya–Obreschkoff or HKO) Theorem says that if $p$ and $q$ are non-constant polynomials with real coefficients. Then $p \prec q$ or $q\prec p$ if and only if, for all $\alpha, \beta \in \mathbb{R}$ such that $\alpha^2+\beta^2 \neq 0$, the polynomial $ \alpha p +\beta q $ has simple, real zeros (see, e.g.~\cite[Theorem 6.3.8]{MR1954841}. An extension to this result to allow for multiple roots and non-strict interlacing can also be found in the literature. For completeness, we present a proof, following the scheme from \cite{MR1179101,MR1268783}. 
In order to simplify the statement, we assume here that a constant polynomial is also real-rooted:
\begin{proposition} 
\label{lem:interlacing}
For $p, q\in \pp_n(\rr)\setminus \pp_{n-1}(\rr)$,  
$$
p \preccurlyeq q \text{ or } q \preccurlyeq p \quad \Leftrightarrow \quad \alpha p+\beta q\in \pp_n(\rr)  \quad \text{for every } \alpha,\beta\in \rr.
%\alpha^2+\beta^2 \neq 0.
$$
Additionally, if $p$ and $q$ are relatively prime (i.e., they do not share a root), then we can replace $\preccurlyeq$ above by $\prec$.
\end{proposition}
\begin{proof}
(Sufficiency)  Without loss of generality, let $p,q$ be monic and assume that $p \preccurlyeq q$, so that their roots, enumerated with account of their multiplicities, satisfy 
$$
\lambda_1(p) \leq \lambda_1(q) \leq \lambda_2(p) \leq \lambda_2(q) \leq \cdots \leq  \lambda_n(p) \leq \lambda_n(q).
$$
% and let $M$ be the smallest distance between any two \textit{distinct} roots of the polynomial $p q$ of degree $2n$. For $j\in \N$, define
For $\varepsilon>0$, define 
$$
p_\varepsilon(x):= \prod_{k=1}^n \left(x-\lambda_{k,\varepsilon}(p) \right), \qquad  q_\varepsilon(x):= \prod_{k=1}^n \left(x-\lambda_{k,\varepsilon}(q) \right),
$$
with
$$
\lambda_{k,\varepsilon}(p):=  \lambda_k(p)+ (2k-2)\varepsilon ,  \qquad  \lambda_{k,\varepsilon}(q)  :=  \lambda_k(q)+  (2k-1)\varepsilon , \qquad k=1,\dots, n.
$$
By construction,
$$
\lambda_{1,\varepsilon}(p) < \lambda_{1,\varepsilon}(q) < \lambda_{2,\varepsilon}(p) < \lambda_{2,\varepsilon}(q) < \cdots <  \lambda_{n,\varepsilon}(p) < \lambda_{n,\varepsilon}(q),
$$
so that $p_\varepsilon \prec q_\varepsilon$.
By the HKO Theorem, for $\alpha, \beta\in \rr$, $\alpha^2+\beta^2 \neq 0$, we have that $\alpha p_\varepsilon+\beta q_\varepsilon\in \pp_n(\rr)$. Since $\lim_{\varepsilon\downarrow 0 }p_\downarrow=p$ and $\lim_{\varepsilon\downarrow 0}q_\downarrow=q$, we conclude that
$$
\alpha p+\beta q=\lim_{\varepsilon\downarrow 0}\, (\alpha p_\varepsilon+\beta q_\varepsilon)\in \pp_n(\rr)\quad \text{for } \alpha,\beta\in \rr, 
\; \alpha^2+\beta^2 \neq 0.
$$

(Necessity) Let $p, q$ be of degree $n$ with real zeros, such that neither $p \preccurlyeq q$ nor $q \preccurlyeq p$ holds. 

If we assume that $p$ and $q$ are relatively prime, this means at least one of the following two possibilities:
\begin{enumerate}[(i)]
    \item there are two consecutive zeros (without loss of generality, of $p$), say $\lambda_i(p)<\lambda_{i+1}(p)$, such that $q(x)\neq 0$ for $\lambda_i(p)<x<\lambda_{i+1}(p)$. We can assume that both $p,q>0$ for $\lambda_i(p)<x<\lambda_{i+1}(p)$. Then $p/q$ achieves its global minimum in $[\lambda_i(p),\lambda_{i+1}(p)]$ at an interior point $\lambda_i(p)<x_0<\lambda_{i+1}(p)$. Denote the value of this minimum by $m>0$, and consider the polynomial 
$$
r_\varepsilon(x):= p(x) - (m+\varepsilon) q(x), \quad \varepsilon\in \R.
$$
By construction, $r_0(x)<0$ for $\lambda_i(p)<x<\lambda_{i+1}(p)$, $x\neq x_0$, and $r_0(x_0)=0$. Hence, $x_0$ is a root of even order of $r_0$: there exists $k\in \N$ such that $r_0^{(j)}(x_0)=0$, $j=0, 1, \dots, 2k-1$, and $r_0^{(2k)}(x_0)\neq 0$.\footnote{\, Here $r^{(j)}(x)$ denotes the $j$-th derivative of $r$ with respect to $x$.} Expanding the polynomial $r_\varepsilon$ at $x=x_0$, we have that
\begin{align*}
r_\varepsilon(x) & = \sum_{j=0}^n \frac{r_0^{(j)}(x_0)}{j!}(x-x_0) - \varepsilon \sum_{j=0}^n \frac{q^{(j)}(x_0)}{j!}(x-x_0)  \\
&= \frac{r_0^{(2k)}(x_0)}{(2k)!}(x-x_0)^{2k} \left( 1+ \mathcal O(x-x_0)\right) - \varepsilon q(x_0) \left( 1+ \mathcal O(x-x_0)\right),
\end{align*}
 so that the roots $\lambda(r_\varepsilon)$ of $r_\varepsilon$ in the neighborhood of $x=x_0$ have the asymptotic expansion
$$
\lambda(r_\varepsilon) = x_0 + \left(\frac{(2k)!\,  q(x_0)}{r_0^{(2k)}(x_0)} \, \varepsilon \right)^{1/(2k)} \left( 1+ o(1)\right), \quad \varepsilon\to 0.
$$
It shows that the polynomial $r_\varepsilon$ cannot have all its roots real for all small values of $\varepsilon\in \R$, which proves the necessity in this situation.

\item there are two zeros (without loss of generality, of $p$), such that  $\lambda_i(p)=\lambda_{i+1}(p)\neq \lambda_i(q)$, which means that for $x_0=\lambda_i(p)$, $q(x_0)\neq 0$, and there is $2\le k\le n$ such that for $j<k$, $p^{(j)}(x_0)= 0$, and $p^{(k)}(x_0)\ne 0$. 

Considering again 
$$
r_\varepsilon(x):= q(x) -  \varepsilon p(x), \quad \varepsilon\in \R;
$$
and expanding it at $x=x_0$, we have that 
\begin{align*}
r_\varepsilon(x) & = \sum_{j=0}^n \frac{p^{(j)}(x_0)}{j!}(x-x_0) - \varepsilon \sum_{j=0}^n \frac{q^{(j)}(x_0)}{j!}(x-x_0)  \\
&= \frac{p^{(k)}(x_0)}{k!}(x-x_0)^{k} \left( 1+ \mathcal O(x-x_0)\right) - \varepsilon q(x_0) \left( 1+ \mathcal O(x-x_0)\right),
\end{align*}
This means that the roots $\lambda(r_\varepsilon)$ of $r_\varepsilon$ in the neighborhood of $x=x_0$ have the asymptotic expansion
$$
\lambda(r_\varepsilon) = x_0 + \left(\frac{k!\,  q(x_0)}{r_0^{(k)}(x_0)} \, \varepsilon \right)^{1/k} \left( 1+ o(1)\right), \quad \varepsilon\to 0,
$$
showing once again that the polynomial $r_\varepsilon$ cannot have all its roots real for all small values of $\varepsilon\in \R$, which proves the necessity in this situation.
\end{enumerate}

Finally, we are ready to drop the assumption that $p$ and $q$ are relatively prime: denote by $r$ the maximum common divisor of $p$ and $q$, necessarily real-rooted, and denote $p=r \widetilde p$, $q=r \widetilde q$, where $\widetilde p$ and $\widetilde q$ are of the same degree and relatively prime. The assumption that neither $p \preccurlyeq q$ nor $q \preccurlyeq p$ holds implies that  $\widetilde p \not \preccurlyeq \widetilde q$,  $\widetilde q \not \preccurlyeq \widetilde p$. Then, as follows from the analysis of cases (i)--(ii) of the proof, not every linear combination $\alpha\widetilde p+\beta\widetilde q$,  $\alpha,\beta\in \rr$, such that $\alpha\widetilde p+\beta\widetilde q \not \equiv 0 $, is real-rooted. It remains to notice that the same conclusion applies then to  
$$
\alpha  p+\beta  q = r\left( \alpha\widetilde p+\beta\widetilde q\right),
$$
and the statement is proved.
\end{proof}

From here and the linearity of the free finite convolution we easily obtain the following interlacing-preservation property:
\begin{proposition}[Preservation of interlacing]
\label{lem:preservinginterlacingMult}
For $p, \tilde p, q\in \pp_n(\rr)\setminus \pp_{n-1}(\rr)$,  
if  $p \preccurlyeq \widetilde{p}$, then %for a  polynomial $q$ of degree $n$,
$$
q \in \P_n(\R) \quad  \Rightarrow \quad p\boxplus_n q \preccurlyeq \widetilde{p} \boxplus_n q
$$
and 
$$
q \in \P_n(\R_{\ge 0}) \quad  \Rightarrow \quad p\boxtimes_n q \preccurlyeq \widetilde{p} \boxtimes_n q.
$$
%The same statements hold if we replace all inequalities by strict inequalities, and all $\preccurlyeq$ by $\prec$.
\end{proposition}

\begin{proof}
Since $p\preccurlyeq\tilde{p}$, then by Proposition \ref{lem:interlacing}, $\alpha p+\beta \tilde{p}\in \pp_n(\rr)$ for every $\alpha,\beta\in \rr$. By assertions \textit{(i)} and \textit{(ii)} of Proposition~\ref{prop:realrootedness}, 
\begin{align*}
  \alpha (p  \boxdot_n  q) + \beta (\tilde{p}   \boxdot_n  q) & = (\alpha p+\beta \tilde{p}) \boxdot_n  q\in \pp_n(\rr), 
\end{align*}
where $ \boxdot_n $ represents either $\boxplus_n$ or $ \boxtimes_n $.
Thus, again by Proposition \ref{lem:interlacing}, either $p  \boxdot_n  q\preccurlyeq \tilde{p}  \boxdot_n  q $ or $\tilde{p}  \boxdot_n  q\preccurlyeq  p \boxdot_n  q$.

However, by \eqref{zerossymmetric}, $p\preccurlyeq\tilde{p}$ implies that $e_1(p) / e_0(p)\leq e_1(\tilde{p})/ e_0(\tilde{p})$. In the case of $\boxplus_n$, 
$$
\frac{e_1(p\boxplus_n q)}{e_0( p \boxplus_n q) }
=\frac{e_1(p)e_0( q ) +e_1(q) e_0( p ) }{e_0( p ) e_0(q)}
=\frac{e_1(p)  }{e_0( p )  } + \frac{e_1(q)  }{e_0( q )  } \leq \frac{e_1(\tilde p)  }{e_0( \tilde p )  } + \frac{e_1(q)  }{e_0( q )  } = \frac{e_1(p\boxplus_n q)}{e_0( \tilde p \boxplus_n q) } ,
$$ 
which implies that $p\boxplus_n q\preccurlyeq \tilde{p} \boxplus_n q $.

Analogously, if  $q\in \pp_n(\rr_{\ge 0})$, we have that $e_1(q)/e_0(q)\geq 0$, so that 
$$
\frac{e_1(p\boxtimes_n q)}{e_0(p\boxtimes_n q)} = \frac{1}{n}\, \frac{e_1(p) e_1(q)}{e_0(p ) e_0(q)}\leq \frac{1}{n}\, \frac{e_1(\tilde p) e_1(q)}{e_0(\tilde p ) e_0(q)} = \frac{e_1(\tilde p\boxtimes_n q)}{e_0(\tilde p\boxtimes_n q)},
$$ which, once again, means that $p\boxtimes_n q\preccurlyeq \tilde{p} \boxtimes_n q $. 
\end{proof}
\begin{remark}\label{preserveinterlacing.negativezeros}
Noticing that $p \preccurlyeq q$ if and only if $q(-x) \preccurlyeq p(-x)$ and that $p(x) \boxtimes_n q(-x)=[p\boxtimes_n q] (-x)$, we can easily extend Proposition \ref{lem:preservinginterlacingMult} to include polynomials with negative real roots. Namely, if $p, \widetilde{p}\in \monicpolreal$ and $p \preccurlyeq \widetilde{p}$, then for a polynomial $q$ of degree $n$, 
$$
q \in \P_n(\R_{\le 0}) \quad  \Rightarrow \quad \widetilde{p} \boxtimes_n q \preccurlyeq p \boxtimes_n q.
$$
\end{remark}

\begin{definition}
Given $p\in\P_n$ with $n\geq 2$, we define its (absolute) \textbf{root separation} or \textbf{mesh} as the minimal distance between its roots:
$$
\mesh (p)=\min \{|\lambda_i(p)-\lambda_j(p)| : 1\leq i <j \leq n\}
$$
(see e.g.~\cite{Branden16}).
\end{definition}
Since $\mesh(p)>r$ if and only if $p(x)\preccurlyeq p(x-r)$, and $p(x-r)=p(x) \boxplus_n(x-r)^n$, it follows that the preservation of interlacing implies the zero separation does not decrease under finite additive free convolution:
\begin{proposition}[Preservation of mesh]
\label{lem:preserving.mesh}
If $p,q\in \P_n(\R)$, both of degree $n$,  then 
$\mesh(p\boxplus_n q) \geq \mesh(p)$.
\end{proposition}
\begin{proof}
If $\mesh(p)>r$, then $p\preccurlyeq [p \boxplus_n(x-r)^n]$, and by preservation of interlacing, $[p\boxplus q] \preccurlyeq [p \boxplus_n(x-r)^n \boxplus_n q]$. Since the second polynomial is simply $[p\boxplus q](x-r)$, we conclude that $\mesh(p\boxplus q)>r$.
\end{proof}

A useful corollary of this result is
\begin{corollary} \label{cor:meshincreas}
If all roots of $p\in \P(\rr)$ are simple and if $q$ is the (unique) inverse of $p$ under additive convolution, then $q\notin \P(\rr)$. 
\end{corollary}

Finite free multiplicative convolution has similar properties, now using the relative separation or logarithmic mesh:
\begin{definition}
Given $p\in \pp(\rr_{>0})$ with $n\geq 2$, we define its \textbf{logarithmic mesh} as the minimal ratio (bigger than 1) between its roots:
$$
\lmesh(p)=\min \{ \lambda_i(p)/\lambda_{i+1}(p)  : 1\leq i <  n\},
$$
assuming $\lambda_1(p)\geq \lambda_2(p)\geq \dots \geq \lambda_n(p)>0$.
\end{definition}

Again, the observation that for $r>0$,  $\lmesh(p)>r$ if and only if $p(x) \preccurlyeq [p(x) \boxtimes_n(x-r)^n]$, yields
\begin{proposition}[Preservation of lmesh, \cite{MR2755692}]
\label{lem:preserving.lmesh}
If $p,q\in \pp(\rr_{>0})$, then
$\lmesh(p\boxtimes q) \geq \lmesh(p)$.
\end{proposition}
\begin{proof}
If $\lmesh(p)>r$ then $p\preccurlyeq [p \boxtimes_n(x-r)^n]$, and by preservation of interlacing, $[p\boxtimes q]\preccurlyeq [p \boxtimes_n(x-r)^n \boxtimes_n q]=[p\boxtimes q] \boxtimes_n(x-r)^n$, so $\lmesh(p\boxtimes q)>r$.
\end{proof}
 
Again, as in the case of Corollary~\ref{cor:meshincreas}, we have  
\begin{corollary} \label{cor:Lmeshincreas}
If all roots of $p\in \P(\R_{> 0})$ are simple and if $q$ is an inverse of $p$ under multiplicative convolution, then $q\notin \P(\rr_{>0})$. 
\end{corollary}

\section{ Convolutions of hypergeometric polynomials} \label{sec:convolutionHGP}

\subsection{Finite free multiplicative convolution}\

We start with a simple result that will allow us to ``assemble'' more complicated hypergeometric polynomials from elementary ``building blocks'' using the free multiplicative convolution:
\begin{theorem} \label{thm:multiplicativeConvHG}
If $n\in \mathbb Z_{\ge 0}$, and  
$$p(x)=\pFq{i_1+1}{j_1}{-n, \bm a_1}{\bm b_1}{x}, \qquad q(x)=\pFq{i_2+1}{j_2}{-n, \bm a_2}{\bm b_2}{x},$$
where the parameters $\bm a_1,\bm a_2,\bm b_1,\bm b_2 $ are tuples (of sizes $i_1,i_2,j_i,j_2$, respectively), then their $n$-th free multiplicative convolution is given by
$$[p\boxtimes_n q](x)=\pFq{i_1+i_2+1}{j_1+j_2}{-n, \bm a_1, \bm a_2}{\bm b_1,\bm b_2}{x}.$$
\end{theorem}
\begin{proof}
For a hypergeometric polynomial
$$
\pFq{i+1}{j}{-n,\ \bm{a}}{\bm b}{x}= \raising{\bm b}{n}\sum_{k=0}^n \frac{\falling{n}{k} \raising{\bm a}{k}}{\raising{\bm b}{k}} \frac{x^k(-1)^k}{k !},
$$
written in the form \eqref{monicP}, the coefficient $e_k$ was given in \eqref{leadingcoeffHGFBis}, that is,
\begin{equation}
    \label{expressionCOeffHGP}
e_k= (-1)^{n }   \binom{n}{k}
    \raising{\bm a}{n-k} \raising{\bm b+n-k}{k},   
\end{equation}  
and the assertion is a straightforward application of the formula in \eqref{coeffMultConv}.
\end{proof} 
A simple consequence is that for $m\in \nn$,
$$
\left( \pFq{i+1}{j}{-n, \bm a}{\bm b}{x}\right) ^{\left( \boxtimes_n\right) m} = \pFq{im+1}{jm}{-n, \bm a,\bm a, \dots, \bm a}{\bm b,\bm b, \dots, \bm b}{x},
$$
where we denote
$$
p ^{ \left( \boxtimes_n\right) m } = \underbrace{p \boxtimes_n \dots \boxtimes_n p}_{m \text{ times}}.
$$

\begin{remark} \label{rem:Multinverses}
Notice that if in Theorem \ref{thm:multiplicativeConvHG}  some entries of the tuple $\bm a_1$ coincide with the corresponding entries of the tuple $\bm b_2$, then they will appear in both the numerator and the denominator of $p\boxplus_n q$, and we can cancel them to obtain the ${}_{i_1+i_2+1-k}F_{j_1+j_2-k}$ hypergeometric polynomial where $k$ is the number of canceled entries. 

In particular, 
$$
\pFq{i+1}{j}{-n, \bm a}{\bm b}{x} \boxtimes_n \pFq{j+1}{i}{-n, \bm b}{\bm a}{x} = \pFq{1}{0}{-n }{\cdot }{x} = (1-x)^n
$$
(see \eqref{Power1F0}), or in the terminology of Definition~\ref{def:multInverse}, 
$$
\pFq{i+1}{j}{-n, \bm a}{\bm b}{x}  
$$
is an inverse of 
$$
    \pFq{j+1}{i}{-n, \bm b}{\bm a}{x} 
$$
under the multiplicative convolution.
\end{remark}

\subsection{Finite free  additive convolution} \label{sec:additive}\ 

Since the definition of additive convolution \eqref{coeffAdditiveConv}  is substantially less straightforward than \eqref{coeffMultConv}, we will use the alternative expression \eqref{thirdDefAdditiveConv} in terms of the differential operator $D_p$ introduced in \eqref{diffOperatorforP}. Using \eqref{connectionRandL} and \eqref{expressionCOeffHGP}, we see that for
$$
p(x)=\pFq{i+1}{j}{-n,\ \bm{a}}{\bm b}{x} 
$$
the corresponding differential operator is  
$$
D_p= e_0(p) \sum_{k=0}^n (-1)^k \frac{\falling{\bm b+n-1}{k} }{\falling{\bm a+n-1}{k}k!} \left(\frac{\partial}{\partial x}\right)^k = e_0(p) \sum_{k=0}^n (-1)^{k(i+j+1)} \frac{\raising{-\bm b-n+1}{k} }{\raising{-\bm a-n+1}{k}k!} \left(\frac{\partial}{\partial x}\right)^k,
$$
with the coefficient $e_0(p) = (-1)^n \raising{\bm a}{n} $ given in \eqref{leadingcoeffHGF}. 
By \eqref{hypergeoSeries}, this is, up to a constant, a terminating hypergeometric series $_i F_j$ evaluated in $(-1)^{i+j+1} \partial / \partial x$, which yields: 
\begin{lemma}
    \label{lemma:DiffOperator}
 If parameters $\bm a$ and $ \bm b  $ are tuples (of sizes $i$ and $j$, respectively), whose entries satisfy restrictions~\eqref{assumptionsHGFNew}, then
    \begin{equation}
    \label{D_p}
    \pFq{i+1}{j}{-n,\ \bm a}{\bm b}{x}= (-1)^n \raising{\bm a}{n}  \, \HGF{j}{i}{-\bm b-n+1}{-\bm a-n+1}{(-1)^{i+j+1}\frac{\partial}{\partial x}} [x^n].
\end{equation}
\end{lemma}

Equivalently, by \eqref{correspondencediffOperatorforP2}, 
\begin{equation} \label{D_p3}
    \HGF{j}{i}{\bm a}{\bm b}{\frac{\partial}{\partial x}} [x^n]= \frac{(-1)^{jn}}{\raising{\bm b}{n}} \pFq{i+1}{j}{-n,\ -\bm b-n+1}{-\bm a-n+1}{(-1)^{i+j+1}x}.
\end{equation}

Using this result in \eqref{thirdDefAdditiveConv} we get
\begin{theorem}
\label{thm.additive.conv.HG}
Let $p$ and $q$ be hypergeometric polynomials of the following form:
$$
p(x)= \pFq{i_1+1}{j_1}{-n,\ \bm a_1}{\bm b_1}{x}, \qquad q(x)=\pFq{i_2+1}{j_2}{-n,\ \bm a_2}{\bm b_2}{x},
$$
where the parameters $\bm a_1,\bm a_2,\bm b_1,\bm b_2 $ are tuples (of sizes $i_1,i_2,j_i,j_2$, respectively). 

Then, with the notation \eqref{leadingcoeffHGF}, their additive convolution $p\boxplus_n q(x)$ is given by
\begin{equation*}
    \begin{split}
      \raising{\bm a_1}{n} \raising{\bm a_2}{n}  \, \HGF{j_1}{i_1}{-\bm b_1-n+1}{-\bm a_1-n+1}{(-1)^{i_1+j_1+1}\frac{\partial}{\partial x}}  \HGF{j_2}{i_2}{-\bm b_2-n+1}{-\bm a_2-n+1}{(-1)^{i_2+j_2+1}\frac{\partial}{\partial x}} [x^n].
    \end{split}
\end{equation*} 
\end{theorem}

Theorem~\ref{thm.additive.conv.HG} shows that factorization identities (or summation formulas) for hypergeometric functions lead to a representation of the corresponding polynomials in terms of the additive convolution of simpler components:
\begin{corollary}
\label{cor.additive.conv.as.poruduct.HG}
Assume that
\begin{equation}
    \label{summationformula}
    \HGF{j_1}{i_1}{\bm a_1}{\bm b_1}{x}  \HGF{j_2}{i_2}{\bm a_2}{\bm b_2}{ x}=\HGF{j_3}{i_3}{\bm a_3}{\bm b_3}{x},
\end{equation}
and let 
$$
p(x)=\pFq{i_1+1}{j_1}{-n,\ -\bm b_1-n+1}{-\bm a_1-n+1}{ x} 
$$
and
$$
q(x)=\pFq{i_2+1}{j_2}{-n,\ -\bm b_2-n+1}{-\bm a_2-n+1}{ x}.
$$
Then the additive convolution of the hypergeometric polynomials
$$
p\left((-1)^{i_1+j_1} x \right)\boxplus_n q\left((-1)^{i_2+j_2} x \right) 
$$
is, up to a constant factor, equal to 
$$
\pFq{i_3+1}{j_3}{-n,\ -\bm b_3-n+1}{-\bm a_3-n+1}{(-1)^{i_3+j_3}  x}.
$$
\end{corollary}
\begin{proof}
    By \eqref{D_p}, 
    \begin{align*}
        p(x) &=   (-1)^n \raising{-\bm b_1-n+1}{n}  \, \HGF{j_1}{i_1}{ \bm a_1}{ \bm b_1}{(-1)^{i_1+j_1+1}\frac{\partial}{\partial x}} [x^n],
    \end{align*}
    so that
    $$
    p\left((-1)^{i_1+j_1+1} x \right) =  (-1)^{n(i_1+j_1)} \raising{-\bm b_1-n+1}{n}  \, \HGF{j_1}{i_1}{ \bm a_1}{ \bm b_1}{\frac{\partial}{\partial x}} [x^n],
    $$
    with an analogous formula valid for $q$. Hence, by Theorem~\ref{thm.additive.conv.HG}, 
    $$
    p\left((-1)^{i_1+j_1+1} x \right)\boxplus_n q\left((-1)^{i_2+j_2+1} x \right)
    $$ 
    is equal to
    $$
    \HGF{j_1}{i_1}{ \bm a_1}{ \bm b_1}{ \frac{\partial}{\partial x}} \HGF{j_2}{i_2}{ \bm a_2}{ \bm b_2}{ \frac{\partial}{\partial x}}[x^n] = \HGF{j_3}{i_3}{\bm a_3}{\bm b_3}{\frac{\partial}{\partial x}}[x^n],
    $$
    up to a multiplicative factor. Now by \eqref{D_p3} this is a constant multiple of 
    $$
    \pFq{i_3+1}{j_3}{-n,\ -\bm b_3-n+1}{-\bm a_3-n+1}{(-1)^{i_3+j_3+1}  x},
    $$
    which yields the identity
    $$
    p\left((-1)^{i_1+j_1+1} x \right)\boxplus_n q\left((-1)^{i_2+j_2+1} x \right) = C \pFq{i_3+1}{j_3}{-n,\ -\bm b_3-n+1}{-\bm a_3-n+1}{(-1)^{i_3+j_3+1}  x}.
    $$
    Formula \eqref{inversion} allows to reduce it to 
    $$
    p\left((-1)^{i_1+j_1} x \right)\boxplus_n q\left((-1)^{i_2+j_2} x \right) = C \pFq{i_3+1}{j_3}{-n,\ -\bm b_3-n+1}{-\bm a_3-n+1}{(-1)^{i_3+j_3}  x},
    $$
    as claimed; the value of the constant $C$ can be obtained by examining the leading coefficients in the identity above:
    $$
    C=(-1)^{n(j_1+j_2+j_3+1)} \frac{\raising{\bm b_1}{n} \raising{\bm b_2}{n} }{\raising{\bm b_3}{n}}.
    $$
   Observe, however, that the summation formula \eqref{summationformula} implies that the parameters  $\bm b_j$ satisfy certain algebraic relations, which in practice simplifies the expression for the constant $C$ considerably.
\end{proof}
\begin{remark} \label{remarkParity}
\begin{enumerate}[(i)]
    \item
    In the case when $i_1+j_1$ and $i_2+j_2$ have equal parity, \eqref{inversion} and Corollary \ref{cor.additive.conv.as.poruduct.HG} yield an expression for the free convolution $p \boxplus_n q$, see Examples \ref{example37}--\ref{example3.12} below. 

\item The direct computation of additive convolution via \eqref{coeffAdditiveConv} for hypergeometric polynomials produces formulas where the coefficients can be expressed in terms of the hypergeometric functions evaluated at $\pm 1$. The approach of Corollary \ref{cor.additive.conv.as.poruduct.HG} presents fewer but more elegant formulas, which is the reason for our choice.

\item  Additional examples can be obtained from known summation formulas that involve evaluation in constant multiples or powers of the variable $x$.
Although these cases are not covered by Corollary \ref{cor.additive.conv.as.poruduct.HG}, we can still use Theorem \ref{thm.additive.conv.HG} and similar arguments for further examples of additive convolution of hypergeometric polynomials (now evaluated in $cx$, $x^2$, etc.). In particular, this allows us to study the symmetrization $p(x)\boxplus_n p(-x)$ of several hypergeometric polynomials $p$. Observe that symmetrization is an instance of non-linear transformation of the original polynomial $p$; first non-trivial examples of such transformations that preserve real zeros appeared in the work of Branden \cite{MR2831515}. We are planning to address these issues in detail in future work. 
\end{enumerate}
\end{remark}

As an illustration, we will analyze several factorization identities for hypergeometric functions (a good source is Chapter 2 of \cite{grinshpan2013generalized}) and their consequences for finite additive convolution. Our intention is not to be exhaustive; instead, we concentrate on the most revealing or less trivial formulas.

\begin{example}[Additive convolution of two Laguerre polynomials] \label{example37}
By binomial identity, we know that 
$$
\HGF{1}{0}{c_1}{\cdot}{x}  \HGF{1}{0}{c_2}{\cdot}{ x}=\HGF{1}{0}{c_1+c_2}{\cdot}{x}.
$$
Using Corollary \ref{cor.additive.conv.as.poruduct.HG} we obtain,
$$
\pFq{1}{1}{-n}{b_1}{x} \boxplus_n \pFq{1}{1}{-n}{b_2}{x} =(-1)^n \pFq{1}{1}{-n}{b_1+b_2+n-1}{x}.
$$
By Equation \eqref{HGFLaguerre}, this can be rewritten in terms of Laguerre polynomials as
$$
L_n^{(\alpha)}(x) \boxplus_n L_n^{(\beta)}(x) =\frac{(-1)^n}{n!} L_n^{(\alpha+\beta+n+1)}(x).
$$
\end{example}

\begin{example} 
\label{example38}
Euler's transformation (see identity (10) in \cite{grinshpan2013generalized}) is  
$$
\HGF{1}{0}{c_1+c_2-d}{\cdot}{x}  \HGF{2}{1}{d-c_1, d-c_2}{d}{ x}=\HGF{2}{1}{c_1,\ c_2}{d}{x}.
$$
Using Corollary \ref{cor.additive.conv.as.poruduct.HG} and some straightforward simplifications, we obtain
$$
\pFq{1}{1}{-n}{b_1+b_2-a}{x} \boxplus_n \pFq{2}{2}{-n,\ a}{a-n+1-b_1,\ a-n+1-b_2}{x} =(-1)^n \pFq{2}{2}{-n,\ a}{b_1,\ b_2}{x}.
$$
In terms of Laguerre polynomials we can write it as
$$
L_n^{(b_1+b_2-a-1)}(x) \boxplus_n \pFq{2}{2}{-n,\ a}{a-n+1-b_1,\ a-n+1-b_2}{x} =\frac{(-1)^n}{n!} \pFq{2}{2}{-n,\ a}{b_1,\ b_2}{x}.
$$
\end{example}

\begin{example} \label{example39}
Clausen's formula (see identity (11) in \cite{grinshpan2013generalized}) asserts that
$$
\left[  \HGF{2}{1}{c,\ d}{c+d+1/2}{x}\right]^2=\HGF{3}{2}{2c,\ 2d,\ c+d}{c+d+1/2,\  2c+2d}{x}.
$$
With an appropriate change of parameters, Corollary \ref{cor.additive.conv.as.poruduct.HG} yields that with
\begin{equation}
    \label{eq:pForexample3.10}
p(x)=\pFq{2}{2}{-n,a+b-1/2}{a - \frac{n-1}{2},\ b- \frac{n-1}{2}}{x},
\end{equation}
we have that for
\begin{equation}
    \label{eq:ConvForexample3.10}
    p(x) \boxplus_n p(x)= (-1)^n \frac{\raising{a+b-1/2}{n}}{\raising{2a+2b+n-1}{n}}\, \pFq{3}{3}{-n,\ a+b-1/2,\ 2a+2b+n-1}{2a,\ 2b,\ a+b}{x}.
\end{equation}
\end{example}

\begin{example}
\label{exm:2F0.plus.2F0}
Identity (18) in \cite{grinshpan2013generalized}, related to the product of Bessel functions, is 
$$
\HGF{0}{1}{\cdot}{c}{x}  \HGF{0}{1}{\cdot}{d}{x}=\HGF{2}{3}{\frac{c+d}{2},\ \frac{c+d-1}{2}}{c,\ d,\ c+d-1}{4x}.
$$
By Corollary \ref{cor.additive.conv.as.poruduct.HG} and formula \eqref{correspondencediffOperatorforP2}, the additive convolution $p\boxplus_n q$ of the hypergeometric polynomials (closely related to Bessel polynomials, defined in \eqref{reversedLaguerre})
\begin{equation}
    \label{besselPexample311}
    p(x)=\pFq{2}{0}{-n,\ -c-n+1}{\cdot}{x}, \qquad q(x)=\pFq{2}{0}{-n,\ -d-n+1}{\cdot}{x},
\end{equation}
is given by
$$
\frac{(-4)^n   }{ \raising{c+d-1}{n}}\, \pFq{4}{2}{-n,\ -c-n+1,\ -d-n+1,\ -c-d-n+2}{-\frac{c+d}{2}-n+1,\ -\frac{c+d-1}{2}-n+1}{x/4}.
$$

We can simplify this expression by setting $a=-c-n+1$ and $b=-d-n+1$, so that
$$
\pFq{2}{0}{-n,\ a}{\cdot}{x} \boxplus_n \pFq{2}{0}{-n,\ b}{\cdot}{x} = \frac{(-4)^n   }{ \raising{a+b+n}{n}}\,  \pFq{4}{2}{-n,\ a,\ b,\ a+b+n}{\frac{a+b}{2},\ \frac{a+b+1}{2}}{x/4}.
$$
In particular, with $a=b$, we get
$$
\pFq{2}{0}{-n,\ a}{\cdot}{x} \boxplus_n \pFq{2}{0}{-n,\ a}{\cdot}{x} = \frac{(-4)^n   }{ \raising{2a+n}{n}}\,  \pFq{3}{1}{-n,\ a, \ 2a +n}{ a+\frac{ 1}{2}}{x/4}.
$$
\end{example}

\begin{example} \label{example3.12}
By identity (17) in \cite{grinshpan2013generalized},
$$
\HGF{1}{0}{2c-2d}{\cdot }{x}  \HGF{3}{2}{2d-1,\ d+1/2,\ d-c-1/2}{c+d+1/2,\ d-1/2}{x}=\HGF{3}{2}{2c-1,\ c+1/2,\ c-d-1/2}{c+d+1/2,\ c-1/2}{x}.
$$
This implies that the additive convolution $p\boxplus_n q$ of the hypergeometric polynomials
\begin{equation}
    \label{p11}
    p(x)=\pFq{1}{1}{-n}{2d-2c-n+1}{x}  
\end{equation}
and 
$$
 q(x)=\pFq{3}{3}{-n,\ -c-d-n+1/2,\ -d-n+3/2}{-2d-n+2,\ -d-n+1/2,\ c-d-n+3/2}{x}
$$
is  
$$
\frac{(-1)^n \raising{d-1/2}{n}}{\raising{c-1/2}{n}}\, \pFq{3}{3}{-n,\ -c-d-n+1/2,\ -c-n+3/2}{-2c-n+2,\ -c-n+1/2,\ d-c-n+3/2}{x}.$$
In other words, additive convolution of $q$ with the polynomial $p$ in \eqref{p11} swaps the parameters $c$ and $d$ in $q$, up to a multiplicative constant.
\end{example}

\section{Real zeros of hypergeometric polynomials} \label{sec:compilation}

A representation of a hypergeometric polynomial as a finite free convolution of more elementary blocks combined with the properties of preservation of the real zeros and interlacing of the free convolutions (see Section~\ref{sec:realrootednessfreeconv}) is an effective tool that allows us to analyze when all roots of a specific hypergeometric polynomial are real\footnote{\, As it was mentioned, some of the results below established by the multiplicative free convolution can have alternative proofs in the contexts of ﬁnite multiplier sequences.}.

In order to use this tool, we need to create an inventory of the simplest hypergeometric polynomials with real (or positive) roots that will serve as basic building blocks for more complicated functions.

\subsection{Simplest real hypergeometric polynomials }\
\label{ssec:simplest.hp}

For small values of $i$ and $j$, the cases when
$$
\pFq{i+1}{j}{-n, \bm a}{\bm b }{x}
$$
has only real roots are well studied and follow from the explicit expressions appearing in Section~\ref{sec:classicalHGP}. For example, for $_1F_0$ this is a consequence of formula \eqref{Power1F0},  while for the $_1F_1$ case we can use the connection with the Laguerre polynomials  \eqref{HGFLaguerre},  whose zeros are well understood. In particular, it follows that all roots of
$$ 
p(x)=\pFq{1}{1}{-n}{b}{x},
$$
which up to a constant coincides with the Laguerre polynomial $L_n^{(b-1)}$, are positive only when $b>0$; they are non-negative if we also admit the values $b\in (-\zz_n)$ (in this case, the polynomial has zeros at the origin with multiplicity $-b+1$, see \eqref{reductionLaguerre}), and $p\in \P_n(\R)$ also if $b\in (-1,0)$, see e.g. \cite[\S 6.73]{szego1975orthogonal}. 
  
Several results on the zero interlacing of Laguerre polynomials can be found in \cite{arvesu2021interlacing, driver2016interlacing, driver2019sharp}. For instance, for  $\alpha>-1$, 
$$ 
L_n^{\alpha}\preccurlyeq L_n^{\alpha+t}, \quad %L_n^{\alpha}\preccurlyeq L_{n-1}^{\alpha},
0 \leq t\leq 2;
$$
also, if $\alpha+1\geq n$, then  
$$ 
L_n^{\alpha}\preccurlyeq L_n^{\alpha+3}.
$$

These facts have the following translation in terms of the ${}_1 F_1$ hypergeometric polynomials: for $b >0$,
\begin{equation}
    \label{interlacing1F1}
     \pFq{1}{1}{-n}{b}{x}\preccurlyeq \pFq{1}{1}{-n}{b+t}{x}, \quad 0\leq t\leq 2;
\end{equation}
also, if $b \geq n$, then
\begin{equation}
    \label{interlacing1F1bis}
    \pFq{1}{1}{-n}{b}{x}\preccurlyeq \pFq{1}{1}{-n}{b+3}{x}.
\end{equation}
 
\medskip

In the case of Bessel or reciprocal Laguerre polynomials,  by \eqref{reversedLaguerre}, 
\begin{equation}
    \label{dualityLaguerre2}
 \pFq{2}{0}{-n,\   a}{ \cdot }{x}=
x^n \pFq{ 1}{1}{-n }{ -n -{  a} +1}{ -\frac{1}{x}} = n! \,  x^n   L_n^{(-n-a)}(-1/x),
\end{equation}
which shows that the zeros of
$$
\pFq{2}{0}{-n,\   a}{ \cdot }{x}
$$
are negative when $a< -n+1$, and real when $-n+1<a <-n+2$. Furthermore, since for $-n-a-1\in (-\infty,-1)\setminus(-\zz_{n})$ we know that $L_n^{(-n-a)}$ has all its roots different from 0, and has at least one complex root, we conclude that for $a>-n+2$ (and $a\notin(-\zz_{n})$), not all zeros of 
$$
\pFq{ 2}{0}{-n, \ a}{ \cdot  }{ x}
$$
are real. Finally, recall that in the remaining cases, when $a\in (-\zz_{n})$, we obtain a real rooted polynomial of degree $-a$ (smaller degree than $n$).

From the interlacing property \eqref{interlacing1F1} we readily deduce that   for $0\leq t\leq 2$ and $a< -n+1$ one has
\begin{equation}
    \label{interlacing2F0}
    \pFq{2}{0}{-n,\   a}{ \cdot }{x} \preccurlyeq  \pFq{2}{0}{-n,\ a-t}{ \cdot }{x}.
\end{equation}  

We summarize in Table \ref{tab:1F1} the real-rooted cases discussed so far.

\begin{table}[h]
    \centering    
\begin{tabular}{|c|c|c|c|}
\hline $\mathbf{a}$ & $\mathbf{b}$ & \textbf{Roots in} & \textbf{Comments} \\
\hline $\cdot$ & $\cdot$ & $\{1\}$ & Identity $(1-x)^n$ \\
\hline $\cdot$ & $\R_{>0}$ & $\R_{>0}$ & Standard Laguerre \\
\hline $\cdot$ & $(-1,0)$ & $\R_{>-1}$ & Non-standard Laguerre \\
\hline $\cdot$ & $- \Z_n$ & {$\R_{\geq 0}$} & Laguerre root mult $-b+1$ at 0 \\
\hline $\R_{<-n+1}$ & $\cdot$ & $\R_{<0}$ & Bessel \\
\hline $(-n+1,-n+2)$ & $\cdot$ & $\mathbb{R}$ & Bessel  \\
%\hline $\{-n+1, \ldots,-1,0\}$ & $\cdot$ & $\R_{<0}$ & Bessel of degree ($-a$ ) \\
\hline
\end{tabular}
\vspace{2mm}
\caption{Real zeros of $\pFq{i+1}{j}{-n, \bm a}{\bm b}{x}$ for $0\le i+j \le 1$. By standard Laguerre we understand $L_n^{(\alpha)}$, given in \eqref{laguerreP}, with $\alpha>-1$. Here $\R_{<c}:= (-\infty,c)$,  $\R_{>c}:= (c,+\infty )$, and $\R_{\ge c}:= [c,+\infty )$. }
    \label{tab:1F1}
\end{table}

\medskip

Finally, by \eqref{JacobiPfla1}--\eqref{JacobiPfla2}, the case of ${}_2 F_1$  is equivalent to studying Jacobi polynomials $P_n^{(\alpha,\beta)}$, see, e.g.,~\cite{MR2723248, szego1975orthogonal}. In particular, all pairs of parameters $(\alpha, \beta)$ for which their zeros are real and simple have been described in the literature (see \cite{dominici2013real}). For a small degree, $n=1,2,3$, a complete description appears in \cite[Prop. 4]{dominici2013real}; see \cite[Thm. 5]{dominici2013real} for higher degrees.  Formulas \eqref{reductionJacobi1}--\eqref{reductionJacobi1bis} also give us the cases where polynomials have multiple roots (which can occur only at $\pm 1$). 

All these facts can be transferred to hypergeometric polynomials $_2F_1$ using identities \eqref{HPJacobiPfla1}--\eqref{reduction2F14bis}. In Table~\ref{tab:2F1}  we summarize the known information, identifying when the roots are all non-positive, all non-negative, or when we have at least one positive and one negative root.

\begin{remark}
\label{rem.multiple.roots}
A particular interesting case follows from \eqref{reduction2F1} by taking $k=1$: for every $b\in \rr \setminus\{-n\}$ one has
\begin{equation}
\label{eq.curious.b.bplus1}
\pFq{2}{1}{-n,\ b+1}{ b}{x}= (1-x)^{n-1} \left(b-x(b+n)\right)=(-1)^n (b+n) (x-1)^{n-1}(x-\tfrac{b}{b+n}) .
\end{equation}
Notice that the only root of the polynomial in \eqref{eq.curious.b.bplus1} that is not at $x=1$ is negative if $-n<b<0$, and positive if $b<-n$ or $b>0$. Using this expression, a direct computation yields that given a real rooted polynomial $p(x)$, % one gets 
$$
p(x) \boxtimes_n \pFq{2}{1}{-n,\ b+1}{ b}{x} =  (-1)^n \raising{b+1}{n-1} \left(  b \, p(x) +  x p'(x) \right),$$
see for instance \cite[Lemma 3.5]{arizmendi2021finite} for similar type of computation.

Moreover, since $\pFq{2}{1}{-n,\ b+1}{ b}{x}$ interlaces the multiplicative identity $\pFq{1}{0}{-n}{.}{x}$ (either to the left or to the right, depending on the sign of $b/(b+n)$), we get in the same fashion that $p(x)$ interlaces $p(x) \boxtimes_n \pFq{2}{1}{-n,\ b+1}{ b}{x}$.

More generally, by \eqref{JacobiPfla1}--\eqref{transfJacobi}, \eqref{reduction2F1} and Table~\ref{tab:2F1}, we have that for $k\in \zz_n$, polynomial
\begin{equation}
\label{eq.curious.b.bplusk}
\pFq{2}{1}{-n,\ b+k}{ b}{x}
\end{equation}
has a root at $x=1$ of multiplicity $n-k$; the remaining roots are all in $[0,1]$ if $b>0$, and bigger than 1 if $b<-n-k+1$.
\end{remark}

Several results on the zero interlacing of Jacobi polynomials can be found in \cite{driver2008interlacing}. For instance, for $\alpha,\beta>-1$,
$$
P_{n}^{(\alpha+t,\beta)} \preccurlyeq P_{n}^{(\alpha,\beta+s)}, \quad 0\leq t,s\leq 2;
$$
actually, in both cases the interlacing is strict ($\prec$), unless $t=s=0$. This fact has the following translation in terms of the ${}_2 F_1$ hypergeometric polynomials:  
\begin{itemize}
    \item for $b>0$ and $a> n+b $,
    \begin{equation}
        \label{interlacing2F1_1}
        \pFq{2}{1}{-n,a+s}{b}{x}\preccurlyeq \pFq{2}{1}{-n,a+t}{b+t}{x}, \quad 0\le s, t \le 2;
    \end{equation}
    
    \item for $b>0 $ and $a<-n+1$,
    \begin{equation}
        \label{interlacing2F1_2}
        \pFq{2}{1}{-n,a}{b+t}{x}\preccurlyeq \pFq{2}{1}{-n,a-s}{b}{x}, \quad 0\le s, t \le 2;
    \end{equation}

    \item For $b<a-n+1$ and $a<-n+1$,
    \begin{equation}
        \label{interlacing2F1_3}
        \pFq{2}{1}{-n,a-t}{b-t}{x}\preccurlyeq \pFq{2}{1}{-n,a}{b-s}{x}, \quad 0\le s, t \le 2.
    \end{equation}
\end{itemize}

\begin{table}[h]
    \centering
\begin{tabular}{|c|c|c|c|}
\hline $\mathbf{a}$ & $\mathbf{b}$ & \textbf{Roots in} & \textbf{Ref.} \\
\hline $ \R_{<-n+1}$ & $(-\Z_n)\cup\R_{\geq 0}$ & $\R_{\leq 0}$ & {\cite[Thm 1]{dominici2013real}}, \eqref{2ndreduction2F1} \\
\hline $\{b+1, b+2, \dots\} \cup \R_{>b+n-2}  $ & $ (-\Z_n)\cup\R_{\geq 0} $ & $\R_{\geq 0}$ & {\cite[Thm 1,8]{dominici2013real}}, \eqref{2ndreduction2F1}-\eqref{reduction2F14bis}  \\
\hline $\R_{<-n+1}$ & $\R_{<a-n+2} \cup \{a-1, a-2, \dots\} $ & $\R_{>0}$ & {\cite[Thm 1,8]{dominici2013real}}, \eqref{reduction2F1} \\
\hline $(-n+1,-n+2)$ & $(-\Z_n) \cup \R_{>-1}  $ & $\mathbb{R}$ & {\cite[Thm 6,8]{dominici2013real}}, \eqref{2ndreduction2F1} \\
\hline $(-n+1,-n+2)$ & $  \R_{<a-n+2} \cup\{a-1, a-2, \dots \} $ & $\mathbb{R}$ & {\cite[Thm 6,7]{dominici2013real}}, \eqref{reduction2F1} \\
\hline $\R_{<-n+1} \cup \R_{>b+n-2}$  & $(-1,0)$ & $\R$  & {\cite[Thm 6,7]{dominici2013real}}\\
\hline
\end{tabular}
\vspace{2mm}
\caption{Real zeros of $\pFq{2}{1}{-n,  a}{ b}{x}$. Recall that we always assume \eqref{assumptionsHGFNew}, so that $
  a \notin (-\Z_n)$, and the polynomial is of degree exactly $n$. Moreover, a zero at $x=0$ appears only when $b \in (-\Z_n)$. 
}
    \label{tab:2F1}
\end{table}

\subsection{General hypergeometric polynomials}
\label{ssec:General.hypergeometric.polynomials}

In this section, we will use the tools of Theorems~\ref{thm:multiplicativeConvHG} and \ref{thm.additive.conv.HG} and Corollary \ref{cor.additive.conv.as.poruduct.HG}. 
 
For instance, a finite multiplicative convolution with a Laguerre polynomial yields the following result, which shows that we can add a positive parameter downstairs without affecting the real-rootedness of a hypergeometric polynomial. Moreover, if the parameters we add differ in less than 2, then we get interlacing and monotonicity:
\begin{theorem}
    \label{prop:addingbelow}
Let $\bm a\in \R^i$, $\bm b \in \R^j$, and $\gamma >0$.  Then
    \begin{equation}
        \label{addingBelow}
        \pFq{i+1}{j}{-n,\ \bm a}{ \bm b}{x} \in \P_n(\R)
     \quad \Longrightarrow \quad 
    \pFq{i+1}{j+1}{-n,\ \bm a}{ \bm b, \gamma}{x} \in \P_n(\R).
    \end{equation}
    Moreover, if  $0\leq t\leq 2$,  
    $$
    \pFq{i+1}{j}{-n,\ \bm a}{ \bm b}{x} \in \P_n(\R_{\geq0}) \quad \Longrightarrow \quad \pFq{i+1}{j+1}{-n,\ \bm a}{ \bm b, \gamma}{x}\preccurlyeq\pFq{i+1}{j+1}{-n,\ \bm a}{ \bm b, \gamma+t}{x}.
    $$
\end{theorem}
\begin{proof}
    This is a straightforward consequence of the formulas in Theorem~\ref{thm:multiplicativeConvHG} with the explicit expression \eqref{HGFLaguerre}, the interlacing \eqref{interlacing1F1}, and the properties of the multiplicative convolution stated in Section~\ref{sec:realrootednessfreeconv}.
\end{proof}

\begin{remark}
\label{rem:addingbelow}
In the previous proposition, we just illustrate the case that might be more useful, but there are several other interesting side results that either follow from the proposition or are its slight modifications. We collect some of these interesting facts next:
\begin{enumerate}[(i)]
\item First, the addition of a positive parameter $\gamma$ below not only preserves the class $\P_n(\rr)$, but also $\P_n(\rr_{\geq 0})$ and $\P_n(\rr_{\leq 0})$. 

\item Also, notice that instead of ``adding'' a positive parameter downstairs, we can also ``remove'' a positive parameter from upstairs (or even ``move'' it from upstairs to downstairs). For instance, with the same hypothesis of Theorem \ref{prop:addingbelow}, we obtain that
    \begin{equation}
        \label{addingAbove}
    \pFq{i+2}{j}{-n,\ \bm a, \gamma}{ \bm b }{x} \in \P_n(\R) \quad \Longrightarrow \quad 
    \pFq{i+1}{j}{-n,\ \bm a}{ \bm b}{x} \in \P_n(\R).
    \end{equation}
This follows from the fact that if we already have a parameter $\gamma$ upstairs, then by Theorem \ref{prop:addingbelow} we can add a parameter $\gamma$ downstairs and then cancel both. 

\item  The interlacing preservation by the multiplicative convolution with a Laguerre polynomial can be stated in a more general form. Namely, if $\bm a_1\in \R^{i_1}$, $\bm a_2\in \R^{i_2}$, $\bm b_1 \in \R^{j_1}$, $\bm b_2 \in \R^{j_2}$  are such that for the real-rooted hypergeometeric polynomials, the interlacing 
$$\pFq{i_1+1}{j_1}{-n,\ \bm a_1}{ \bm b_1}{x} \preccurlyeq \pFq{i_2+1}{j_2}{-n,\ \bm a_2}{ \bm b_2}{x} $$ 
holds, then, for $\gamma>0$,
$$
\pFq{i_1+1}{j_1+1}{-n,\ \bm a_1}{ \bm b_1, \gamma}{x} \preccurlyeq \pFq{i_2+1}{j_2+1}{-n,\ \bm a_2}{ \bm b_2, \gamma}{x}.
$$

\item Finally, notice that so far we have only been concerned with the case where we multiply by a standard Laguerre polynomial, corresponding to Row 2 in Table \ref{tab:1F1}. However, we can adapt Theorem \ref{prop:addingbelow} (and its modifications that we just mentioned) to include Rows 3 and 4. The reader should be aware that in some cases we need stronger assumptions on the hypergeometric polynomial. For example, if we want to use the non-standard Laguerre polynomials in row 3 in Table \ref{tab:1F1}, which are not in $\P_n(\R_{\geq 0})$, then the other polynomial must belong to $\P_n(\R_{\geq 0})$ rather than $\P_n(\R)$.
\end{enumerate}
\end{remark}

Although Theorem \ref{prop:addingbelow} is rather limited, it already covers useful results from the literature. Notice, for instance, that the well-know fact that 
    $$
    \sum_{k=0}^n a_k x^k\in \P_n(\R) \quad \Rightarrow  \quad  \sum_{k=0}^n \frac{a_k}{k !} x^k\in \P_n(\R)
    $$
    (see \cite[Theorem 2.4.1]{polya1914uber} or \cite[Problem V.1.65]{polya1976problems}) is just a particular case of \eqref{addingBelow} with $\gamma=1$. 

    In the same fashion, we can use convolution with Bessel polynomials (see Table~\ref{tab:1F1} and interlacing \eqref{interlacing2F0}) to derive the following result:
\begin{theorem}
    \label{prop:convBessel}
    Let $\bm a\in \R^i$, $\bm b \in \R^j$, and $\gamma <-n+1$.
    Then
    \begin{equation}
        \label{convBessel1}
        \pFq{i+1}{j}{-n,\ \bm a}{ \bm b}{x} \in \P_n(\R)
     \quad \Longrightarrow \quad 
    \pFq{i+2}{j}{-n,\ \bm a, \gamma}{ \bm b}{x} \in \P_n(\R).
    \end{equation}
    Moreover, if   $0\leq t\leq 2$, then 
    $$
      \pFq{i+1}{j}{-n,\ \bm a}{ \bm b}{x} \in \P_n(\R_{\ge 0})
     \quad \Longrightarrow \quad 
     \pFq{i+2}{j}{-n,\ \bm a, \gamma}{ \bm b}{x}\preccurlyeq\pFq{i+2}{j}{-n,\ \bm a, \gamma-t}{ \bm b}{x}.$$
\end{theorem}

Following the same reasoning of Remark \ref{rem:addingbelow} used to extend Theorem \ref{prop:addingbelow}, we can adapt Theorem \ref{prop:convBessel} to produce related results. For example, we can remove a parameter $\gamma <-n+1$ from downstairs, or we can preserve the interlacing of the given polynomial. We can also obtain a similar result using Rows 6 and 7 of Table \ref{tab:1F1} instead of just Row 5. We avoid the details for the sake of brevity.

Another example of how the free multiplicative convolution with Bessel polynomials allows us to give a straightforward proof of a known result is as follows: by \eqref{laguerreP},
    $$
   r(x)= x^n L^{(0)}_{n}(1/x)= \sum_{k=0}^{n} (-1)^k \binom{n}{k} \frac{1}{k!}\,   x^{n-k}.
$$
Hence, if 
$$
p(x)= \sum_{k=0}^n a_k x^k= \sum_{k=0}^n (-1)^k \left( (-1)^k\, a_{n-k}\right) x^{n-k}
$$
and 
$$
q(x)= \sum_{k=0}^n b_k x^k= \sum_{k=0}^n (-1)^k \left( (-1)^k\, b_{n-k}\right) x^{n-k},
$$
then
\begin{align*}
    p(x) \boxtimes_n q(x) \boxtimes_n r(x) & =    \sum_{k=0}^n (-1)^k \binom{n}{k}^{-1} \frac{a_{n-k} \, b_{n-k}}{k!}\,   x^{n-k} 
   = \frac{(-1)^n}{n!}\, \sum_{k=0}^n    k!\,  a_{k} \, b_{k} \,   (-x)^{k}. 
\end{align*}
We have that $r\in \P_n(\R_{>0})$; if  additionally,  $p, q\in \P_n(\R_{\leq 0})$, then by Proposition \ref{prop:realrootedness} (and the ``rule of sign'' formulated there), we have that $p(x) \boxtimes_n q(x) \boxtimes_n r(x) \in \P_n(\R_{\ge 0})$. Thus,  
$$
p, q \in\R_{\leq 0} \quad  \Rightarrow \quad \sum_{k=0}^{n} k !\,  a_k b_k x^k \text{ has only real zeros},
$$
which is a weaker version of Schur's theorem (where the assumptions are that $p, q \in \P_n(\R)$  and $ a_k, b_k \ge  0$   for all $k$); it was first proved in \cite{schur1914zwei}, and appears also in \cite[Problems V.2.155-156]{polya1976problems}. 

\medskip 

If we now consider the free multiplicative convolution with Jacobi polynomials, we get
\begin{theorem}
    \label{prop:adding.above.below}
    Let $\bm a\in \R^i$, $\bm b \in \R^j$,  $\beta>0$, and $\alpha>\beta +n-1$. Then
    $$
    \pFq{i+1}{j}{-n,\ \bm a}{ \bm b}{x} \in \P_n(\R)
     \quad \Longrightarrow \quad 
    \pFq{i+1}{j+1}{-n,\ \bm a,\ \alpha}{ \bm b,\ \beta}{x} \in \P_n(\R),$$
Moreover, for $0\leq t,s\leq 2$:
$$ 
  \pFq{i+1}{j}{-n,\ \bm a}{ \bm b}{x} \in \P_n(\R_{\ge 0})
     \quad \Longrightarrow \quad 
     \pFq{i+1}{j+1}{-n,\ \bm a,\ \alpha+s}{ \bm b,\ \beta}{x} \preccurlyeq  \pFq{i+1}{j+1}{-n,\ \bm a,\ \alpha+t}{ \bm b,\ \beta+t}{x}.$$ 
\end{theorem}
For the last assertion, we used interlacing properties \eqref{interlacing2F1_1}--\eqref{interlacing2F1_3}. Following the same reasoning of Remark \ref{rem:addingbelow}, used to extend Theorem \ref{prop:addingbelow}, we can adapt Theorem \ref{prop:adding.above.below} to yield related results. For example, the previous result corresponds to multiplication with a specific type of Jacobi polynomial, corresponding to row 2 of Table \ref{tab:2F1}, but we can also obtain a similar result using the other rows of Table \ref{tab:2F1}.

If we consider Laguerre, Bessel, and Jacobi polynomials as building blocks, and iteratively apply Theorems \ref{prop:addingbelow}, \ref{prop:convBessel}, or \ref{prop:adding.above.below}, we can directly prove that a large class of hypergeometric polynomials (with several parameters) is real-rooted, or even more, their roots are all positive or all negative. Here are some illustrations.

\begin{theorem} \label{prop:42bis}
For any $i, j\geq 0$, if $b_1,\dots,b_j>0$ and  $a_1,\dots,a_{i}<-n+1$ then the roots of the hypergeometric polynomial
$$p(x):={}_{i+1} \mathcal F_j\left(\begin{matrix} -n, \bm a \\ \bm b 
\end{matrix};x\right),
$$
with $\bm a = \left( a_1,\dots,a_{i} \right)$, $\bm b = \left( b_1,\dots,b_{j} \right)$, are all real and have the same sign. Specifically, if $i$ is even, the roots are all positive, and if $i$ is odd, the roots are all negative. 

Furthermore, if additionally $0\leq t \leq 2$ and $i$ is even, then 
    $$\pFq{i+1}{j}{-n,\ \bm a}{ b_1,\ \dots,\ b_j}{x}\preccurlyeq\pFq{i+1}{j}{-n,\ \bm a}{ b_1+t,\ b_2,\ \dots,\ b_j}{x}, $$
and 
 $$
  \pFq{i+1}{j}{-n,\ a_1-t,\ a_2,\ \dots,\ a_{i}}{ \bm b}{x}
 \preccurlyeq \pFq{i+1}{j}{-n,\ a_1,\ \dots,\ a_{i}}{ \bm b}{x}
.$$
If $i$ is odd, both interlacings should be reversed.
\end{theorem}
\begin{proof}
By Theorem~\ref{thm:multiplicativeConvHG}, 
$$
p(x)= \pFq{2}{0}{-n,\  a_1}{\cdot }{x} \boxtimes_n  \dots \boxtimes_n \pFq{2}{0}{-n,\  a_i}{\cdot }{x}\boxtimes_n \pFq{1}{1}{-n,}{ b_1}{x} \boxtimes_n  \dots \boxtimes_n \pFq{1}{1}{-n,}{ b_j}{x} .
$$
By Table~\ref{tab:1F1}, rows 2 and 5, the first $i$ polynomials in the product are Bessel polynomials with all negative roots, while the the last $j$ polynomials in the product are Laguerre polynomials with all positive roots. By the rule of signs from Section \ref{sec:realrootednessfreeconv} applied several times, we find that all roots of $p$ have the same sign, and the sign depends on the parity of $i$, the number of Bessel polynomials that we multiply. 

The first interlacing result follows from the same factorization, but grouping all terms except the one with $b_1$:
$$
p(x)= \pFq{1}{1}{-n}{b_1}{x} \boxtimes_n  \pFq{i+1}{j-1}{-n,\ \bm a}{ b_2,\ \dots,\ b_j}{x}.
$$
Then, Theorem \ref{prop:addingbelow} with $\gamma=b_1$ yields the desired result. Similarly, the second interlacing result follows from Theorem \ref{prop:convBessel} with $\gamma=a_1$. Of course, by symmetry, the interlacing results hold when we vary any given parameter in the polynomial, and not only $a_1$ or $b_1$.
\end{proof}

The following proposition is established in a similar way, but now we require the use of Jacobi polynomials.
\begin{theorem} \label{prop:42}
If $j\geq i$, $b_1,\dots,b_j > 0$ and $a_1,\dots,a_{i}\in \rr$ such that $a_s\geq n-1+b_s$ for $s=1,\dots,i$, then the hypergeometric polynomial 
$$p(x)=
\pFq{i+1}{j}{-n,\ \bm a}{ \bm b}{x} ,
$$
with $\bm a = \left( a_1,\dots,a_{i} \right)$, $\bm b = \left( b_1,\dots,b_{j} \right)$, 
has all positive roots.
\end{theorem}
\begin{proof}
First, we consider the case of $i=j$. Then by Theorem~\ref{thm:multiplicativeConvHG}, 
$$
p(x):= \pFq{2}{1}{-n,\  a_1}{ b_{1}}{x} \boxtimes_n  \dots \boxtimes_n \pFq{2}{1}{-n,\  a_i}{ b_{i}}{x} 
$$
with $b_1,\dots,b_j,\in (0,\infty)$, and $a_s\geq n-1+b_s$ for $s=1,\dots,i$. By \eqref{HPJacobiPfla1} (see also Table~\ref{tab:2F1}, second row), this is a multiplicative convolution of polynomials with zeros in $(0,1)$, so that all roots of $p$ are positive.

Now, if $j>i$, the assertion follows from what we just established by invoking~\eqref{addingBelow}. 
\end{proof}
\begin{remark}
Similarly to Theorem \ref{prop:42bis}, we can obtain interlacing results by varying the parameters of the polynomial. Since this factorization involves Jacobi polynomials, we need to use Theorem \ref{prop:adding.above.below} and vary the parameters accordingly. 
\end{remark}

Our final observation in this section is that, by taking the reciprocal polynomial, the study of ${}_{i+1}F_j$ polynomials is in a certain sense dual to the study of ${}_{j+1}F_i$ polynomials. 
\begin{remark}
\label{rem:pFq.reversed}
Given a polynomial $p\in \pp_n(\cc)$, its \textbf{reciprocal polynomial} is defined formally by 
\begin{equation}
    \label{eq.def.reversed}
\hat{p}(x):=x^n p(1/x).
\end{equation}
Alternatively, its coefficients are given by $e_j(\hat{p}) = (-1)^n e_{n-j}(p),$ for $j=0, 1, \dots, n.$

It is not difficult to check that when restricted to $\pp_n(\cc\setminus \{0\})$, the reciprocation operation is actually an involution that maps a polynomial with roots $\lambda_1,\dots, \lambda_n$, to a polynomial with roots $1/\lambda_1,\dots, 1/\lambda_n$. In particular, the sets $\pp_n(\rr\setminus \{0\})$, $\pp_n(\rr_{>0})$ and $\pp_n(\rr_{<0})$ are invariant under this operation. Moreover, Lemma \ref{lem:pFq.reversed} has the simple consequence that
$$\pFq{i+1}{j}{-n,\ \bm a}{ \bm b}{x}\qquad \text{ and }\qquad \pFq{j+1}{i}{-n,\ -n -{\bm b} +1}{ -n -{\bm a} +1}{(-1)^{i+j}x}$$
are reciprocal polynomials. This implies that the reciprocation operation defines a bijection from the ${}_{i+1}F_j$ polynomials in $\pp_n(\cc\setminus \{0\})$ to the ${}_{j+1}F_i$ polynomials in $\pp_n(\cc\setminus \{0\})$. More precisely, if we let
$$
p(x)=\pFq{i+1}{j}{-n,\ \bm a}{ \bm b}{x}, \quad q(x)= \pFq{j+1}{i}{-n,\ -n -{\bm b} +1}{ -n -{\bm a} +1}{x},
$$
then
$$
p\in \P_n(\R\setminus\{0\})  \quad \Longleftrightarrow \quad q\in \P_n(\R\setminus\{0\}),
$$
and 
$$
p\in \P_n(\R_{>0})  \quad \Longleftrightarrow \quad q \in \begin{cases}
    \P_n(\R_{>0}) & \text{if $i+j$ is even,} \\ 
     \P_n(\R_{<0}) & \text{if $i+j$ is odd.}
\end{cases}
$$

After this discussion, it should be clear that we can focus our study on polynomials of the type $i\leq j$ and then extrapolate the results to the case $i>j$. Notice that the fact that the properties of the Bessel polynomials follow from those of the Laguerre polynomials is the particular case of this bijection when $i=0$, $j=1$.
\end{remark}

In the following sections, we will discuss in more detail what can be said about ${}_2 F_2$, ${}_3 F_1$ and ${}_3 F_2$ hypergeometric polynomials, since these appear frequently in practical applications, but are much less studied than the ${}_2 F_1$ functions. Once again, we seek to illustrate our approach without aspiring to provide comprehensive results on real-rootedness or interlacing of these polynomials.

\subsection{${}_2 F_2$ and ${}_3 F_1$ Hypergeometric polynomials}\label{sec:2F2real}\

Several conclusions about ${}_2 F_2$ and ${}_3 F_1$ polynomials can be reached using the results in Section \ref{ssec:General.hypergeometric.polynomials}. Some of them are known (and we provide a bibliographic reference whenever available or known to us), and others are apparently new.  

By Remark \ref{rem:pFq.reversed}, the ${}_3 F_1$ polynomials are just reciprocal ${}_2 F_2$ polynomials. Therefore, we will focus on the ${}_2 F_2$ family and discuss some consequences for the ${}_3 F_1$ polynomials at the end of the section.

Table \ref{tab:2F2} summarizes some results that we can obtain directly using additive or multiplicative convolution of our building blocks (e.g. Laguerre, Bessel, and Jacobi). In the following, we provide a brief justification when the result is not trivial. 

\begin{table}[htb]
    \centering
    \begin{tabular}{|c|c|c|c|}
    \hline $\mathbf{a}$ & $\mathbf{b_1}$&$\mathbf{b_2}$ & \textbf{Roots in}  \\
\hline $ \R_{<-n+1}$ & $(- \mathbb Z_n)\cup\R_{> 0}$ & $\R_{>0}$& $\R_{\leq 0}$ \\
\hline $\{b_1+1, b_1+2, \dots\} \cup \R_{>b_1+n-2}  $ & $ (- \mathbb Z_n)\cup\R_{> 0} $& $\R_{>0}$ & $\R_{\geq 0}$   \\
\hline $\R_{<-n+1}$ & $ \R_{<a-n+2} \cup\{a-1, a-2, \dots \}$ & $\R_{>0}$ & $\R_{>0}$  \\
\hline $(-n+1,-n+2)$ & $(- \mathbb Z_n) \cup \R_{>-1}  $& $\R_{>0}$ & $\mathbb{R}$\\
\hline $(-n+1,-n+2)$ & $  \R_{<a-n+2} \cup\{a-1, a-2, \dots \} $& $\R_{>0}$ & $\mathbb{R}$  \\
\hline $\R_{<-n+1} \cup \R_{>b_1+n-2}$  & $(-1,0)$ & $\R_{>0}$& $\R$  \\   \hline 
\hline $k+1/2$  & $2-b_2>0$ or $1-b_2>0$ & $\mathbb{R}_{>0}$ 
& $\mathbb{R}_{> 0}$  \\
\hline $b_1+k-1/2$  & $(b_2+t+1)/2$ & $\mathbb{R}_{>0}$ 
& $\mathbb{R}_{> 0}$  \\
\hline $(b_1+1)/2+k$   & $ 2(b_2-1+t)$ & $\mathbb{R}_{>0}$ 
& $\mathbb{R}_{> 0}$  \\
\hline $b_1 /2+k$   & $ 2(b_2+t)-1$ & $\mathbb{R}_{>0}$ 
& $\mathbb{R}_{> 0}$  \\
\hline $b_2 -1/2 $  & $2b_2-2$  & $(0,1 )$ 
& $\mathbb{R} $  \\
\hline $b_2 -1/2 $  & $2b_2-1$  & $(-1, 0 )$ 
& $\mathbb{R} $  \\
\hline $b_2+k-1/2  $  & $2b_2-2$  & $( 1/2, 1)$ 
& $\mathbb{R} $  \\
\hline $k+1/2$  & $1-b_2$ or $2-b_2$   & $ (-1,0)  $ 
& $\mathbb{R} $  \\
\hline
\end{tabular}
\vspace{2mm}
\caption{Real zeros of $\pFq{2}{2}{-n,\  a}{ b_1,\ b_2}{x}$. The first six rows are a consequence of combining the cases from Table \ref{tab:2F1} and Table \ref{tab:1F1} via the multiplicative convolution identity \eqref{2F2Real1}. For the remaining rows, see Proposition \ref{Prop:410new}. In each appearance in the table above, $k\in \mathbb Z_n $ while $t\in \mathbb Z_n \cup  \rr_{>n-2}$. Recall that we always assume \eqref{assumptionsHGFNew}, so that $  a \notin (-\mathbb Z_n)$, and the polynomial is of degree exactly $n$. Moreover, a zero at $x=0$ appears only when either $ b_i \in (- \mathbb Z_n)$.  }
    \label{tab:2F2}
\end{table}

A particular case of the identity from Theorem~\ref{thm:multiplicativeConvHG} is
\begin{equation} \label{2F2Real1}
    \pFq{2}{2}{-n,\ a}{b_1,\ b_2}{x}=\pFq{2}{1}{-n,\ a}{b_1}{x} \boxtimes_n \pFq{1}{1}{-n}{b_2}{x}.
\end{equation}

Recall that whenever both polynomials on the right-hand side of \eqref{2F2Real1} are in $\P_n(\rr)$ and, additionally, one of them has all its roots of the same sign,  we can conclude that the ${}_2F_2$ polynomial on the left in \eqref{2F2Real1} is also in $\P_n(\rr)$; this is the essence of Theorem \ref{prop:addingbelow} in the case $i=j=1$. Furthermore, we can narrow down the location of its roots to a smaller subset of $\R$ provided that we have additional information on the roots of the two factors on the right.  Therefore, the first six rows of Table \ref{tab:2F2} is a result of combining rows 1--6 of Table \ref{tab:2F1} with row 2 of Table \ref{tab:1F1} via the identity \eqref{2F2Real1}. The other rows are a consequence of the following proposition:
\begin{proposition}
\label{Prop:410new}
Let $n\ge 4$, $k \in \Z_n$, and $t\in \Z_n \cup\rr_{>n-2}$.

The polynomial
\begin{equation*}
   % \label{new2F2}
    \pFq{2}{2}{-n,\ a}{b_1,\ b_2}{x} \in \P_n(\R_{>0})
\end{equation*}
if $b_2>0$, and additionally, one of the following conditions holds:
\begin{enumerate}[(i)]
  \item \label{Prop:410new1}
   $ a=  k + 1/2$ and
   either $b_1=2-b_2>0$ or $b_1=1-b_2>0$;

 \item \label{Prop:410new2} 
    $ a= b_1 + k -1/2$, and 
$b_1 = (b_2+t+1)/2$;

 \item \label{Prop:410new3}
   $ a= (b_1+1)/2+k  $, and $b_1 =2(b_2-1+t)$;

\item \label{Prop:410new4}
   $ a= b_1 /2+k$ and $b_1 =2(b_2+t)-1$.

\vspace{.2cm}

\hspace{-1.3cm} The polynomial  
\begin{equation*}
    \pFq{2}{2}{-n,\ a}{b_1,\ b_2}{x} \in \P_n(\R)
\end{equation*}
if one of the following conditions holds:

\vspace{.2cm}

    \item \label{Prop:410new5} 
    $ a= b_2  -1/2 $,   
$b_1 = 2b_2 -2$, and $b_2\in (0,1)$;

   \item \label{Prop:410new6}
   $ a=  b_2-1/2$, $b_1=2b_2-1$, and $b_2\in (-1,0)$;
   
  \item \label{Prop:410new7} 
    $ a= b_2 + k -1/2 $,  
$b_1 = 2b_2 -2$, and $b_2\in (1/2,1)$;

 \item \label{Prop:410new8}
   $ a=  k +1/2  $, and   $b_1  + b_2\in \{1,2\}$, if   $b_2\in (-1,0) $.
\end{enumerate}    
\end{proposition}
Recall that in practice, $b_1$ and $b_2$ are indistinguishable, and their the roles can be interchanged.  
\begin{proof}
In Example \ref{exm:2F0.plus.2F0} we wrote the polynomial (up to a constant term, and a change of variable $x/4\mapsto x$) 
$$
p(x)=\pFq{4}{2}{-n,\ -c-n+1,\ -d-n+1,\ -c-d-n+2}{-\frac{c+d}{2}-n+1,\ -\frac{c+d-1}{2}-n+1}{x}
$$
as an additive convolution of two Bessel polynomials. 
Recall that the Bessel polynomials \eqref{besselPexample311} in Example \ref{exm:2F0.plus.2F0} 
have real roots for $c,d>-1$ and only negative roots for $c,d>0$. Since by \eqref{leadingcoeffHGF}, for this polynomial $p$,
$$
\frac{e_n(p)}{e_0(p)} = \frac{(-1)^n}{4^n} \frac{\raising{c+d+n-1}{n}}{\raising{c}{n} \raising{d}{n}},
$$
we conclude that  for $c,d>-1$, $c, d \neq 0$, an appropriately normalized $p$ has no roots at the origin and $p\in \P_n(\rr)$. Moreover,  $p\in \P_n(\rr_{<0})$ if $c,d>0$.

As a next step, we can use the reciprocal polynomial as in Remark \ref{rem:pFq.reversed} to transfer these results to a $_3 F_3$ polynomial: namely, we get that 
\begin{equation}
    \label{eq:3F3_1}
    \pFq{3}{3}{-n,\ \frac{c+d}{2}, \ \frac{c+d-1}{2}}{c,\ d,\ c+d-1}{x}
\end{equation}
is real rooted for $c,d>-1$, with all roots positive  for $c,d>0$.

Next, we can get a ${}_2 F_2$ function using the cancellation of some parameters in the hypergeometric expression above; obviously, all conclusions about the location of their zeros remain. Referring to the cases enumerated in the statement of the theorem:
\begin{enumerate}[(a)]
    \item Setting in \eqref{eq:3F3_1} $c=d-1$  we conclude that  $
    \pFq{2}{2}{-n,\ d- 1/2}{d,\ 2d-2}{x}
    $
    is real-rooted for $d>0$, $d\neq 1$, with all roots positive for $d>1$.

    \item With  $c=d$ in \eqref{eq:3F3_1}, we get that $
    \pFq{2}{2}{-n,\ d- 1/2}{d,\ 2d-1}{x}
    $
 is real-rooted for $d>-1$, $d\neq 0$, with all roots positive for $d>0$.
    \item With $c=2-d$ in \eqref{eq:3F3_1}, we see that  $
    \pFq{2}{2}{-n,\ 1/2}{d,\ 2-d}{x}
    $
    is real-rooted for $-1<d<3$, $d\neq 0, 2$, with all roots positive for $0<d<2$.
    \item Finally, $c=1-d$ yield that $
    \pFq{2}{2}{-n,\ 1/2}{d,\ 1-d}{x}
    $
    is real-rooted for $-1<d<2$, $c\neq 0, 1$, with all roots positive for $0<d<1$. 
 \end{enumerate}    
 the assertions in (a) and (b) for the real-rooted case yield \eqref{Prop:410new5} and \eqref{Prop:410new6} of the proposition, respectively. All other assertions can be extended to larger families by ``replacing'' a parameter upstairs by using the finite free multiplicative convolution of the polynomials from above with certain ${}_2 F_1$ polynomials from row 1 of Table \ref{tab:2F1},\footnote{\, Obviously, there are some alternative (sometimes more direct) ways to obtain these results.}
\begin{equation}
    \label{multIdentityCancellation}
    \pFq{2}{2}{-n,\ a_1}{b_1,\ b_2}{x}=\pFq{2}{2}{-n,\ a_2}{b_1, \ b_2}{x} \boxtimes_n \pFq{2}{1}{-n, \ a_1 }{a_2 }{x}.
\end{equation}
Namely, using \eqref{multIdentityCancellation} to combine polynomials from the assertions (a) and (b) above with the hypergeometric polynomial 
$\pFq{2}{1}{-n,\ d+k-1/2}{ d-1/2}{x}$, yields that for $k=1,2,\dots,n-1$,
\begin{enumerate} 
    \item[(a')] \label{2F2-1c} 
    $
    \pFq{2}{2}{-n,\ d+k-1/2}{d,\ 2d-2}{x}
    $
    is real-rooted for $d>1/2$, $d\neq 1$, with all roots positive for $d>1$.

    \item[(b')] \label{2F2-2c} 
    $
    \pFq{2}{2}{-n,\ d+k- 1/2}{d,\ 2d-1}{x}
    $
    has all roots positive for $d>1/2$.
\end{enumerate}   
the real-rooted case in (a') yields \eqref{Prop:410new7} from the proposition. 

Analogously, applying \eqref{multIdentityCancellation} to combine polynomials from (c) and (d) above with the hypergeometric polynomial  $\pFq{2}{1}{-n,\ k+1/2}{ 1/2}{x}$, we conclude that
 \begin{enumerate} 
      \item[(c')]  
    $
    \pFq{2}{2}{-n,\ k+1/2}{d,\ 2-d}{x}
    $
    is real-rooted for $-1<d<3$, $d\neq 0, 2$, with all roots positive for $0<d<2$.
    
    \item[(d')]  
    $
    \pFq{2}{2}{-n,\ k+1/2}{d,\ 1-d}{x}
    $
    is real-rooted for $-1<d<2$, $d\neq 0, 1$, with all roots positive for $0<d<1$.
\end{enumerate}  
These yield \eqref{Prop:410new1} of the proposition, in the case of all positive roots, and  \eqref{Prop:410new8}  otherwise.

In a similar fashion, we can replace some parameters downstairs. For instance, using 
\begin{equation}
\label{multIdentityCancellation2}    
\pFq{2}{2}{-n,\ a_1}{b_1,\ b_2}{x}=\pFq{2}{2}{-n,\ a_1}{b_1, \ b_3}{x} \boxtimes_n \pFq{2}{1}{-n, \ b_3 }{b_2 }{x}
\end{equation}
to combine polynomials from (a') and (b') above  with the polynomials $\pFq{2}{1}{-n,\ 2d-2}{ b}{x}$ and $\pFq{2}{1}{-n,\ 2d-1}{ b}{x}$, all of them satisfying the conditions of row 2 in table \ref{tab:2F1}, yields that for $n\geq 4$, $b>0$ and $k\in \Z_n$, the following families of ${}_2\mathcal{F}_2$ polynomials are real-rooted and their zeros all positive:
\begin{enumerate} 
    \item[(a'')] 
    $
    \pFq{2}{2}{-n,\ d+k-1/2}{d,\ b}{x}
    $
     for $2d-1=b+t$, where either $t\in \zz_n$ or $t>n-2$.
\end{enumerate}
This gives us \eqref{Prop:410new2} of the proposition. If we use \eqref{multIdentityCancellation2} to combine polynomials from (a') and (b') above with $\pFq{2}{1}{d}{ b}{x}$ satisfying the conditions of row 2 in table \ref{tab:2F1}, then we find that for $n\geq 4$, $b>0$ and $k\in\Z_n$, the following families of ${}_2\mathcal{F}_2$ polynomials are real-rooted and all their zeros positive:
\begin{enumerate} 
    \item[(b'')] 
    $
    \pFq{2}{2}{-n,\  d+k-1/2}{2d-2,\ b}{x}
    $
    for $d=b+t$, where either $t\in \zz_n$ or $t>n-2$.

    \item[(c'')] 
    $
    \pFq{2}{2}{-n,\  d+k-1/2}{2d-1,\ b}{x}
    $
for $d=b+t$, where either $t\in \zz_n$ or $t>n-2$.
\end{enumerate}
These imply \eqref{Prop:410new3} and \eqref{Prop:410new4}, respectively.
\end{proof}

\begin{remark}
As we mentioned earlier, our approach is not exhaustive and Table \ref{tab:2F2} does not contain all combinations of parameters that produce real-rooted polynomials. However,  we can use our method to exclude some combinations of parameters.  For example, the contrapositive of (ii) in Proposition \ref{prop:realrootedness} states that $p\boxtimes_n q\notin \P_n(\rr),\ q\in \P_n(\rr_{\geq 0}) \ \Rightarrow \ p\notin \P(\rr)$. We can use it with factorization 
\begin{equation}
    \pFq{2}{2}{-n,\ a}{b_1,\ b_2}{x}\boxtimes_n \pFq{2}{0}{-n,\  b_2}{\cdot }{x}=\pFq{2}{1}{-n,\ a}{b_1}{x}.
\end{equation}
Notice that for $b_2 <-n+1$, the second term on the left-hand side is a Bessel polynomial with all negative roots. Thus, if the pair $(a, b_1)$ does not satisfy some of the conditions of Table \ref{tab:2F1} (so that the right-hand side is a Jacobi polynomial with at least one complex root), then $\pFq{2}{2}{-n,\ a}{b_1,\ b_2}{x}$ also has at least one complex root. This kind of argument allows one to find several major regions of parameters where we can assure that the corresponding polynomials are not real-rooted polynomials. For instance, we have that
$$\pFq{2}{2}{-n,\ a}{b_1,\ b_2}{x} \notin \P(\rr),$$
whenever one of the following assumptions holds:
\begin{itemize}
    \item $b_1, b_2 <-n+1$ and $a >0$ or $a<\max\{b_1-n+2,b_2-n+2\}$; or
    \item $b_1 <-n+1$, $a >0$ and  $b_2>a+n-2$.
\end{itemize}
\end{remark}

We can also draw conclusions about the interlacing of the ${}_2\FF_2$ polynomials using \eqref{interlacing1F1}--\eqref{interlacing1F1bis} and \eqref{interlacing2F1_1}--\eqref{interlacing2F1_3}. For instance, applying Theorem \ref{prop:addingbelow} to the polynomials in Table \ref{tab:2F1} we can obtain the following result that covers several important cases:
\begin{corollary}
Let $c \in (0,\infty)$ and let $a,b\in \rr$ be two parameters such that $\pFq{2}{2}{-n, a}{ b}{x}\in \P_n(\rr_{\ge 0})$ (for instance,  $(a,b)$ belong to a case covered in Table \ref{tab:2F1}). Then 
$$
\pFq{2}{2}{-n,  a}{ b, c}{x} \preccurlyeq \pFq{2}{2}{-n,  a}{ b, c+t}{x}, \quad 0\leq t\leq 2.
$$
If $\pFq{2}{2}{-n, a}{ b}{x}\in \P_n(\rr_{\le 0})$, the interlacing should be reversed.
\end{corollary}

\medskip

Several interesting results on real-rootedness and zero interlacing of ${}_2 F_2$ polynomials have been obtained in \cite[Section 2]{johnston2015quasi}. We revisit them next, using our approach, which allows us to obtain some generalizations.

For example, the fact that for $c>0$, $k\in \zz_n$, and $ b+k \notin (-\zz_n)$, with either $b> 0$, $b < -n-k+1$, or $b\in (-\zz_n)$, then 
$$\pFq{2}{2}{-n,  b+k}{ b, c}{x}\in \P_n(\rr_{\geq 0}),
$$
which follows from rows 2 and 3 of Table \ref{tab:2F2}, partially generalizes the statement of \cite[Theorem 2.2]{johnston2015quasi} (where it is claimed that at least $n-k$ roots are distinct and positive). Moreover, in the particular case of $k=1$, identity \eqref{2F2Real1} reads as  
$$\pFq{2}{2}{-n,  b+1}{ b, c}{x}=\pFq{1}{1}{-n}{c}{x}\boxtimes_n \pFq{2}{1}{-n,  b+1}{ b}{x} \in \P_n(\rr),$$
and we know that all roots are positive if $b\notin [-n,0]$. Moreover, by Remark \ref{rem.multiple.roots} we know that the resulting polynomial interlaces $\pFq{1}{1}{-n}{c}{x}$. These claims are precisely the content of \cite[Theorem 2.3]{johnston2015quasi}. Actually, from the interlacing property of the Laguerre polynomials we can also derive that for $0\leq t \leq 2$ the following interlacing holds:
$$\pFq{2}{2}{-n,  b+1}{ b, c}{x} \preccurlyeq \pFq{2}{2}{-n,  b+1}{ b, c+t}{x}.$$

With a similar procedure, for the case $k=2$ we can factorize 
$$\pFq{2}{2}{-n,  b+2}{ b, c}{x}:=\pFq{1}{1}{-n}{c}{x}\boxtimes_n \pFq{2}{1}{-n,  b+2}{ b}{x}.$$
The conclusions of \cite[Theorem 2.4]{johnston2015quasi} then follow from the properties of the ${}_2 F_1$ polynomial on the right and the multiplicative convolution.

\begin{remark}
    The results in \cite{johnston2015quasi} are based on the quasi-orthogonality property of the corresponding hypergeometric polynomials. Several examples show that a finite free convolution of (quasi-)orthogonal polynomials can generate families of orthogonal or multiple orthogonal polynomials. Understanding this property further is an interesting open problem in this field.
\end{remark}

To finish this section, we use Remark \ref{rem:pFq.reversed} to claim that ${}_3 F_1$ polynomials are just reciprocal ${}_2 F_2$ polynomials. To be more precise, a combination of parameters $(a, b_1, b_2)$ satisfying a condition of any row of Table \ref{tab:2F2} implies that 
\begin{equation}
    \label{reversed3F1}
    \pFq{3}{1}{-n, -  b_1-n+1,- b_2-n+1}{-  a-n+1}{x}
\end{equation}
is real-rooted. We illustrate this assertion in the following Corollary that is a result of taking the reciprocal polynomials from Proposition \ref{Prop:410new}:
\begin{corollary} 
\label{Cor:414new}
Let $n\ge 4$, $k \in \Z_n$, and $t\in \Z_n \cup\rr_{>n-2}$.

The polynomial
\begin{equation*}
    \pFq{3}{1}{-n,\ a_1,\ a_2}{\ b}{x} \in \P_n(\R_{<0})
\end{equation*}
if $a_2<-n+1$, and additionally, one of the following conditions holds:
\begin{enumerate}[(i)]
  \item  
   $ b=-n -k +1/2$ and
   either $a_1=-a_2-2n<-n+1$ or $a_1=-a_2-2n+1<-n+1$;

 \item  
    $ b= a_1 - k +1/2$, and 
    $a_1= (a_2-n-t)/2$;

 \item 
   $ b= (a_1-n)/2-k  $, and $a_1=2a_2+n+1-2t$;

\item  
   $ b= (a_1-n+1)/2-k$ and $a_1=2a_2+n -2t $.

\vspace{.2cm}

\hspace{-1.3cm} Moreover, the polynomial 
\begin{equation*}
    \pFq{3}{1}{-n,\ a_1,\ a_2}{\ b}{x} \in \P_n(\R )
\end{equation*}
if one of the following conditions holds:

\vspace{.2cm}

    \item  
    $ b= a_2+1/2 $,   
$a_1 = 2a_2+n +1$, and $a_2\in (-n ,-n+1) $;

   \item  
   $ b= a_2+1/2 $, $a_1 = 2a_2+n $, and $a_2\in (-n+1,-n+2) $;
   
  \item  
    $ b= a_2 - k +1/2$,  
$a_1 = 2a_2+n  +1$, and $a_2\in (-n,-n+1/2)$;

 \item  
   $ b=-n -k +1/2$, and $a_1+a_2+2n\in \{0, 1\}$, if $a_2\in (-n+1,-n+2)$.
\end{enumerate} 
\end{corollary}

We summarize some of our results on the real zeros of ${}_3 F_1$ polynomials in Table~\ref{tab:3F1}.

\begin{table}[ht]
    \centering
\begin{tabular}{|c|c|c|c|}
\hline $\mathbf{a_1}$ & $\mathbf{a_2}$& $\mathbf{b}$ & \textbf{Roots in} \\
\hline $ \R_{<-n+1}$ & $\R_{<-n+1}$& $ \R_{> 0}$ &  $\R_{> 0}$ \\
\hline  $ \R_{<-n+1}$ & $ \R_{<-n+1}$ & $  \R_{< a_1-n+2} $&  $\R_{< 0}$   \\
\hline $\R_{> b+n-2}$ & $\R_{<-n+1}$& $   \R_{>0}$& $\R_{<0}$  \\
\hline $\R_{<-n+2}\setminus \{- n+1\}$ & $\R_{<-n+1}$ & $(-1,0)  $&  $\mathbb{R}$\\
\hline $\R_{>b+n-2}$ & $\R_{<-n+1}$ &  $ (-1,0) $& $\mathbb{R}$  \\
\hline $(-n+1,-n+2)$ & $\R_{<-n+1}$ & $\R_{<a_1-n+2} \cup \R_{>0}$ & $\R$  \\  \hline 
\hline $\{-a_2-2n, -a_2-2n+1\} \cap \R_{<-n+1}$  & $\R_{<-n+1}$ & $-n -k +1/2$ 
& $\mathbb{R}_{< 0}$  \\
\hline $(a_2-n-t)/2$  & $\R_{<-n+1}$ & $a_1 - k +1/2$ 
& $\mathbb{R}_{< 0}$  \\
\hline $2a_2+n+1-2t$   & $ \R_{<-n+1}$ & $(a_1-n)/2-k$ 
& $\mathbb{R}_{< 0}$  \\
\hline $2a_2+n -2t$   & $ \R_{<-n+1}$ & $ (a_1-n+1)/2-k$ 
& $\mathbb{R}_{< 0}$  \\
\hline $ 2a_2+n +1 $  & $(-n ,-n+1) $  & $a_2+1/2 $ 
& $\mathbb{R} $  \\
\hline $ 2a_2+n  $  & $(-n+1 ,-n+2)$  & $a_2+1/2  $ 
& $\mathbb{R} $  \\
\hline $2a_2+n +1  $  & $(-n ,-n+1/2)$  & $a_2 - k +1/2$ 
& $\mathbb{R} $  \\
 \hline $-a_2-2n+j$, $j\in \{0,1\}$  & $(-n+1, -n+2)  $   & $ -n -k +1/2$ 
 & $\mathbb{R} $  \\
\hline
\end{tabular}
\vspace{2mm}
\caption{Real zeros of $\pFq{3}{1}{-n, \ a_1, \  a_2}{   b}{x}$. 
The first six rows are a consequence of using \eqref{reversed3F1} with the corresponding rows of Table \ref{tab:2F2}.  For the remaining rows, see Corollary \ref{Cor:414new}. In each appearance in the table above, $k\in \mathbb Z_n $ while $t\in \mathbb Z_n \cup  \rr_{>n-2}$. Recall that we always assume \eqref{assumptionsHGFNew}, so that $  a_i \notin (-\Z_n)$, and the polynomial is of degree exactly $n$. Moreover, a zero at $x=0$ appears only when $b \in (-\Z_n)$.   }
    \label{tab:3F1}
\end{table}

\subsection{${}_3 F_2$ Generalized hypergeometric polynomials} \label{sec:3F2real}

Several results on real-rootedness of ${}_3 F_2$ were obtained in the literature, in particular in \cite{Driver/Jordaan, MR2324318, johnston2006properties,johnston2015quasi}. In this section, we focus on how to establish some generalizations of these results using finite free multiplicative convolution (Theorem~\ref{thm:multiplicativeConvHG}).

As in the previous section, we factorize the ${}_3 \FF_2$ polynomial into the multiplicative convolution of two or more real-rooted polynomials, all but one of them with all roots of the same sign. Obviously, these representations are not unique and depend on how we partition the parameter space $\bm a \times \bm b$, with $\bm a=(a_1, a_2)$ and $\bm b=(b_1, b_2)$, into subsets. 

The most basic option is to represent $\bm a \times \bm b$ as the union of $(\{a_j\}, \{\cdot\})$ and $(\{ \cdot\} , \{ b_j\} )$, with $j=1, 2$, which produces a representation of the ${}_3 \FF_2$ polynomial as ${}_2 \FF_0\boxtimes {}_2 \FF_0\boxtimes {}_1 \FF_1\boxtimes {}_1 \FF_1$. To obtain real rooted polynomials, each parameter must satisfy rather restrictive conditions. In general, we can do better by using other partitions, namely:
\begin{itemize}
    \item $(\{a_1\}, \{b_1\}) \cup (\{a_2\}, \{b_2\})$, which produces a representation ${}_2 \FF_1 \boxtimes {}_2 \FF_1$;
    \item $(\{ a_1\}, \{b_1, b_2\}) \cup (\{a_2\}, \{\cdot\})$, which produces a representation ${}_2 \FF_2 \boxtimes {}_2 \FF_0$;
    \item $(\{ a_1, a_2\}, \{b_1\}) \cup (\{\cdot \}, \{b_2\})$, which produces a representation ${}_3 \FF_1 \boxtimes {}_1 \FF_1$.
\end{itemize}

Let us discuss these three partitions more systematically. 

\textbf{Case 1:} we can generate real-rooted polynomials by using the representation
\begin{equation} \label{3F2Real1}
    \pFq{3}{2}{-n,a_1,a_2}{b_1,b_2}{x}=\pFq{2}{1}{-n,a_1}{b_1}{x} \boxtimes_n \pFq{2}{1}{-n,a_2}{b_2}{x}
\end{equation}
and the entries of Table~\ref{tab:2F1}. Namely, we can combine the parameters of any row of Table~\ref{tab:2F1} with those in the rows corresponding to nonnegative (or nonpositive) zeros (rows 1--3). More precise information on the location of the zeros is obtained if we restrict ourselves to rows 1--3 only.

Since the methodology is straightforward, we are not going to provide a comprehensive list of the outcomes of such representations. Instead, we highlight some of the most interesting combinations of parameters $a_i$ and $ b_i$ for which the polynomial on the left-hand side of \eqref{3F2Real1} is real rooted, and point out how they generalize the already known results. Obviously, the roles of $a_1$ and $a_2$ (as well as $b_1$ and $b_2$) can be freely interchanged. 

For instance, by combining the major intervals of parameters in rows 1--3 from Table~\ref{tab:2F1}  we obtain six different domains of parameters for which the zeros of the polynomials are all real and have the same sign: 
\begin{proposition}
    \label{prop:zeros3F2case1cont}
Consider
 $$
 p(x)= \pFq{3}{2}{-n,a_1,a_2}{b_1,b_2}{x}.
 $$
    \begin{enumerate}[(i)]
        \item if $b_1, b_2>0$, $a_1< -n+1$ and $a_2> \min\{ b_1, b_2\} +n-2$ then $p\in \P_n(\R_{<0})$.  
        \item if $b_1, b_2>0$, $a_1> b_1 + n-2$ and $a_2> b_2+n-2$, then $p\in \P_n(\R_{> 0})$.  
        \item \label{eq.item3F2} if $b_1, b_2>0$, $a_1,a_2< -n+1$ then  $p\in \P_n(\R_{>0})$.  
        \item if $a_1, a_2<-n+1$, $b_1>0$ and $b_2< \min\{ a_1, a_2\} -n+2$ then $p\in \P_n(\R_{\leq 0})$.  
        \item if $a_1, a_2<-n+1$, $b_1< a_1 -n+2$ and $b_2< a_2 -n+2$ then $p\in \P_n(\R_{> 0})$.  
        \item if $a_1<-n+1$, $b_1< a_1 -n+2$, $b_2>0$ and $a_2> b_2+n-2$ then $p\in \P_n(\R_{> 0})$.    
    \end{enumerate}
\end{proposition}

The result in \eqref{eq.item3F2} generalizes \cite[Theorem 9]{MR2324318}.  We summarize these six main regions in Table~\ref{tab:3F2Case1major}. 

\begin{table}[h]
    \centering
    \begin{tabular}{|c|c|c|c|c|}
\hline $\mathbf{a_1}$ & $\mathbf{a_2}$  & $ \mathbf{b_1}$  & $ \mathbf{b_2}$  & \textbf{Roots in} \\
\hline $ \R_{<-n+1}$ & $ \R_{>\min \{b_1, b_2 \}+n-2}$ & $\rr_{>0}$ & $\R_{> 0}$ & $\R_{< 0}$  \\ 
\hline $ \R_{>b_1+n-2}$ & $ \R_{> b_2  +n-2}$ & $\rr_{>0}$ & $\R_{> 0}$ & $\R_{> 0}$  \\ 
\hline $ \R_{<-n+1}$ & $\R_{<-n+1}$ & $\rr_{>0}$ & $\R_{> 0}$ & $\R_{> 0}$  \\ 
\hline $ \R_{<-n+1}$ & $ \R_{<-n+1}$ & $\rr_{>0}$ & $ \R_{<\min \{a_1, a_2 \}-n+2}$ & $\R_{< 0}$  \\ 
\hline $ \R_{<-n+1}$ & $ \R_{<-n+1}$ & $ \R_{<a_1-n+2}$ & $ \R_{<a_2-n+2}$ & $\R_{> 0}$  \\ 
\hline $ \R_{<-n+1}$ & $ \R_{> b_2  +n-2}$ & $ \R_{<a_1-n+2}$ & $\R_{> 0}$ & $\R_{> 0}$  \\ 
\hline
\end{tabular}
\vspace{2mm}
\caption{Real zeros of $\pFq{3}{2}{-n, \ a_1, \  a_2}{   b_1, \ b_2}{x}$, as described in Propositions \ref{prop:zeros3F2case1cont}.  Since in this case always $
  a_i \notin (-\Z_n)  $, all polynomials are of degree exactly $n$. 
 }
    \label{tab:3F2Case1major}
\end{table}

In addition to these six domains of parameters, we can obtain more admissible combinations using the remaining values from rows 1--3 in Table~\ref{tab:2F1}. Since the detailed explanation is tedious, our exposition here will be more schematic.

Namely, any assertion from Proposition \ref{prop:zeros3F2case1cont} (equivalently, any row from Table \ref{tab:3F2Case1major}) can be extended as follows:
\begin{enumerate} 
    \item[(R1)] by replacing condition $b_i>0$ by $b \in (-\Z_n)$ we arrive at the same conclusion but with possible roots at the origin. 
    \item[(R2)] by replacing condition $b_i<a_i-n+2$ (or equivalently, $a_i>b_i+n-2$) by $b_i\in \{a_i-1,a_i-2,\dots, a_i-n+1\}$ (or equivalently, $a_i \in \{ b_i+1,b_i+2,\dots, b_i+n-1\}$) we get the same conclusion as before. 
\end{enumerate}
This procedure can be iterated with the other set of parameters as long as the conditions are met. 

Finally, we can apply \eqref{3F2Real1}  combining the parameters from rows 1--3 in Table~\ref{tab:2F1} with those in rows 4--6. In this case, we get real-rooted polynomials, although we cannot guarantee that all the roots will be of the same sign.  Once again, any assertion of Proposition \ref{prop:zeros3F2case1cont} (or the one obtained after applying procedures (R1)--(R2)) can be extended as follows:
\begin{enumerate}
    \item[(R3)] by replacing condition  $b_i>0$ by  $b_i>-1$, or
    \item[(R4)] by replacing condition $a_i<-n+1$ by   $a_i<-n+2$,
\end{enumerate}
which yields real-rooted polynomials. Notice however that procedures (R3)--(R4) cannot be iterated: a second replacement amounts to finding the multiplicative convolution of two real-rooted polynomials (and not all roots are necessarily of the same sign), whose outcome is not determined a priori.  

Let us illustrate the considerations above by retrieving some previously known results.

By applying replacement (R2) to parameters from row 1 in Table \ref{tab:3F2Case1major} we conclude that 
$$
\pFq{3}{2}{-n,a_1,b_2+k}{b_1,b_2}{x} \in \P_n(\R_{<0}) \qquad \text{whenever } b_1, b_2>0,\text{ and }  a_1< -n+1.
$$
This result (together with row 1 in Table \ref{tab:3F2Case1major} itself)   generalizes \cite[Theorem 7]{MR2324318}.

Analogously, applying replacement (R2) to parameters from rows 2, 4--6 in Table \ref{tab:3F2Case1major} (and with an appropriate reparametrization) we obtain that for
$$
p(x)=\pFq{3}{2}{-n,a_1,c+k}{b_1,c}{x}, \quad k=1,\dots, n-1, 
$$
\begin{itemize}
    \item $p \in \P_n(\R_{>0}) $  whenever $ b_1,c >0$,  and  $a_1> b_1 + n-2$.
    \item $p \in \P_n(\R_{<0}) $ whenever $b_1>0$,  $c<-n-k+1$,  and $a_1< -n+1$.
    \item $p \in \P_n(\R_{>0})$ whenever $a_1< -n+1$, $c<-n-k+1$, and $b_1< a_1 - n+2$.
    \item $p \in \P_n(\R_{>0}) $ whenever $c >0$,  $a_1< -n+1$,   and    $b_1< a_1 - n+2$.
    \item $p \in \P_n(\R_{>0})$  whenever $b_1>0$,   $c< -n-k+1$,  and $a_1> b_1 + n-2$.
\end{itemize}
These results partially extend \cite[Corollary 3.2]{johnston2015quasi}.

\medskip

\textbf{Case 2:}  we consider the representation 
\begin{equation} \label{3F2Real2}
    \pFq{3}{2}{-n,\ a_1,\ a_2}{b_1,\ b_2}{x}=\pFq{2}{2}{-n,\ a_1}{b_1, \ b_2}{x} \boxtimes_n \pFq{2}{0}{-n,\ a_2}{\cdot }{x},
\end{equation}
combined with the results from Table~\ref{tab:2F2} and the Bessel polynomials, rows 5--6 from Table~\ref{tab:1F1}. Again, we need one of the polynomials to be real-rooted and the other one to have only nonnegative or only nonpositive roots.

Notice that the first 6 rows of Table~\ref{tab:2F2} were obtained using the representation of the ${}_2 F_2$ polynomials as ${}_2 \FF_1\boxtimes {}_1 \FF_1 $. Therefore, a combination of these rows with rows 5--6 of Table ~\ref{tab:1F1} is equivalent to the factorization ${}_2 \FF_1\boxtimes {}_1 \FF_1 \boxtimes {}_2 \FF_0$, which is already included in case 1. Hence, here we focus only on rows 7--15 of Table~\ref{tab:2F2} (Proposition \ref{Prop:410new}), which were obtained using additive convolution. For instance, a multiplicative convolution with row 5 of Table ~\ref{tab:1F1} yields:
\begin{proposition}
\label{prop:zeros3F2Case2Bis}
Let $n\ge 4$, $k \in \Z_n$, and $t\in \Z_n \cup\rr_{>n-2}$.

The polynomial
\begin{equation*}
    \pFq{3}{2}{-n,\ a_1, \ a_2 }{b_1,\ b_2}{x} \in \P_n(\R_{<0})
\end{equation*}
if $a_2<-n+1$, $b_2>0$, and additionally, one of the following conditions holds:
\begin{enumerate}[(i)]
  \item  
   $ a_1=  k + 1/2$ and
   either $b_1=2-b_2>0$ or $b_1=1-b_2>0$;

 \item   \label{iiOfPropzeros3F2Case2Bis}
    $ a_1= b_1 + k -1/2$, and 
$b_1 = (b_2+t+1)/2$;

 \item  
   $ a_1= (b_1+1)/2+k  $, and $b_1 =2(b_2-1+t)$;

\item  
   $ a_1= b_1 /2+k$ and $b_1 =2(b_2+t)-1$.

\vspace{.2cm}

\hspace{-1.3cm} The polynomial  
\begin{equation*}
    \pFq{3}{2}{-n,\ a_1, \ a_2}{b_1,\ b_2}{x} \in \P_n(\R)
\end{equation*}
if $a_2<-n+1$ and one of the following conditions holds:

\vspace{.2cm}

    \item   
    $ a_1= b_2  -1/2 $,   
$b_1 = 2b_2 -2$, and $b_2\in (0,1)$;

   \item  \label{viOfPropzeros3F2Case2Bis}
   $ a_1=  b_2-1/2$, $b_1=2b_2-1$, and $b_2\in (-1,0)$;
   
  \item  
    $ a_1= b_2 + k -1/2 $,  
$b_1 = 2b_2 -2$, and $b_2\in (1/2,1)$;

 \item  
   $ a_1=  k +1/2  $, and   $b_1  + b_2\in \{1,2\}$, if   $b_2\in (-1,0) $.
\end{enumerate}    
\end{proposition}

Once again, we gather these domains of parameters in Table \ref{tab:3F2case2}. 
 
\begin{table}[hb]
    \centering
    \begin{tabular}{|c|c|c|c|c|}
    \hline $\mathbf{a}_1$ &  $\mathbf{a}_2$ & $\mathbf{b_1}$&$\mathbf{b_2}$ & \textbf{Roots in}  \\
\hline $k+1/2$  & $\mathbb{R}_{< -n+1}$ & $2-b_2>0$ or $1-b_2>0$ & $\mathbb{R}_{>0}$ 
& $\mathbb{R}_{< 0}$  \\
\hline $b_1+k-1/2$ & $\mathbb{R}_{< -n+1}$ & $(b_2+t+1)/2$ & $\mathbb{R}_{>0}$ 
& $\mathbb{R}_{< 0}$  \\
\hline $(b_1+1)/2+k$  & $\mathbb{R}_{< -n+1}$ & $ 2(b_2-1+t)$ & $\mathbb{R}_{>0}$ 
& $\mathbb{R}_{< 0}$  \\
\hline $b_1 /2+k$  & $\mathbb{R}_{< -n+1}$ & $ 2(b_2+t)-1$ & $\mathbb{R}_{>0}$ 
& $\mathbb{R}_{< 0}$  \\
\hline $b_2 -1/2$ & $\mathbb{R}_{< -n+1}$ & $2b_2-2$  & $(0,1 )$ 
& $\mathbb{R} $  \\
\hline $b_2 -1/2$ & $\mathbb{R}_{< -n+1}$ & $2b_2-1$  & $(-1, 0 )$ 
& $\mathbb{R} $  \\
\hline $b_2+k-1/2$ & $\mathbb{R}_{< -n+1}$ & $2b_2-2$  & $(1/2,1)$ 
& $\mathbb{R} $  \\
\hline $k+1/2$ & $\mathbb{R}_{< -n+1}$ & $1-b_2$ or $2-b_2$   & $ (-1,0)  $ 
& $\mathbb{R} $  \\
\hline
\end{tabular}
\vspace{2mm}
\caption{Real zeros of $\pFq{3}{2}{-n,\  a_1, a_2}{ b_1,\ b_2}{x}$, as described in Proposition \ref{prop:zeros3F2Case2Bis}. In each appearance in the table above, $k\in \mathbb Z_n $ while $t\in \mathbb Z_n \cup  \rr_{>n-2}$. Recall that we always assume \eqref{assumptionsHGFNew}, so $  a_i \notin (-\Z_n)$, and the polynomial is of degree exactly $n$. Moreover, a zero at $x=0$ appears only when $b_i \in (-\Z_n)$.
  }
    \label{tab:3F2case2}
\end{table}

A particular case of the assertions \eqref{viOfPropzeros3F2Case2Bis} and (\ref{iiOfPropzeros3F2Case2Bis}) (with $k=t=0$) of Proposition \ref{prop:zeros3F2Case2Bis} is a generalization of \cite[Theorem 8]{MR2324318}. Furthermore, taking   $b_1=b_2=1$ and $k=t=0$ in assertion  (\ref{iiOfPropzeros3F2Case2Bis})  of Proposition \ref{prop:zeros3F2Case2Bis}, we conclude that 
$$
p(x)=\pFq{3}{2}{-n,\ 1/2, \ -n\pm 1/2 }{1,\ 1}{x} \in \P_n(\R_{<0});
$$
the real-rootedness of this polynomial was conjectured by B.~Ringeling and W.~Zudilin\footnote{\, Personal communication.}.

\medskip

\textbf{Case 3: }  
we consider the representation 
\begin{equation} \label{3F2Real1times1F1}
    \pFq{3}{2}{-n,\ a_1,\ a_2}{b_1,\ b_2}{x}=\pFq{3}{1}{-n,\ a_1, \ a_2}{b_1}{x} \boxtimes_n \pFq{1}{1}{-n }{b_2 }{x},
\end{equation}
combined with the results from Table~\ref{tab:3F1} and the Laguerre polynomials, rows 2--3 from Table~\ref{tab:1F1}. Notice that if $b_2>0$, it is sufficient for the ${}_1F_1$ polynomial to be real rooted. On the other hand, using Remark \ref{rem:pFq.reversed} and taking the reciprocal polynomials for both terms in \eqref{3F2Real2} we can show that Case 3 is in a certain sense dual to Case 2. Thus, many of the results below can be derived alternatively using this duality. 

For instance, combining the results of Theorem \ref{prop:addingbelow} and Corollary \ref{Cor:414new} using \eqref{3F2Real1times1F1} we get
\begin{proposition}
\label{prop3F2Case3Newversion}
Let $n\ge 4$, $k \in \Z_n$, and $t\in \Z_n \cup\rr_{>n-2}$.

The polynomial
\begin{equation*}
    \pFq{3}{2}{-n,\ a_1,\ a_2}{ b_1, \ b_2}{x} \in \P_n(\R_{<0})
\end{equation*} 
if $a_2<-n+1$, $b_2>0$, and additionally, one of the following conditions holds:
\begin{enumerate}[(i)]
  \item  
   $ b_1=-n -k +1/2$ and
   either $a_1=-a_2-2n<-n+1$ or $a_1=-a_2-2n+1<-n+1$;

 \item  
    $ b_1= a_1 - k +1/2$, and 
    $a_1= (a_2-n-t)/2$;

 \item  
   $ b_1= (a_1-n)/2-k  $, and $a_1=2a_2+n+1-2t$;

\item  
   $ b_1= (a_1-n+1)/2-k$ and $a_1=2a_2+n -2t $.

\vspace{.2cm}

\hspace{-1.3cm} Moreover, the polynomial 
\begin{equation*}
    \pFq{3}{1}{-n,\ a_1,\ a_2}{\ b}{x} \in \P_n(\R )
\end{equation*}
if $b_2>0$ and one of the following conditions holds:

\vspace{.2cm}

    \item  
    $ b_1= a_2+1/2 $,    $a_1 = 2a_2+n +1$, and $a_2\in (-n ,-n+1) $;

   \item  
   $ b_1= a_2+1/2 $, $a_1 = 2a_2+n $, and $a_2\in (-n+1,-n+2) $;
   
  \item  
    $ b_1= a_2 - k +1/2$,  $a_1 = 2a_2+n  +1$, and $a_2\in (-n,-n+1/2)$;

 \item  
   $ b_1=-n -k +1/2$, and $a_1+a_2+2n\in \{0, 1\}$, if   $a_2\in (-n+1,-n+2)$.
\end{enumerate} 
\end{proposition} 

We summarize these results in Table~\ref{tab:3F2case3}. 

\begin{table}[hb]
    \centering
    \begin{tabular}{|c|c|c|c|c|}
    \hline $\mathbf{a}_1$ &  $\mathbf{a}_2$ & $\mathbf{b_1}$&$\mathbf{b_2}$ & \textbf{Roots in}  \\
\hline $-a_2-2n+j<-n+1$     & $\mathbb{R}_{< -n+1}$ & $-n-k+1/2$ & $\mathbb{R}_{>0}$ 
& $\mathbb{R}_{< 0}$  \\
\hline $(a_2-n-t)/2$  & $\mathbb{R}_{< -n+1}$ & $a_1-k+1/2$ & $\mathbb{R}_{>0}$ 
& $\mathbb{R}_{< 0}$  \\
\hline $ 2a_2+n+1-2t$  & $\mathbb{R}_{< -n+1}$ & $(a_1-n)/2-k$  & $\mathbb{R}_{>0}$ 
& $\mathbb{R}_{< 0}$  \\
\hline $ 2a_2+n-2t$ & $\mathbb{R}_{< -n+1}$ & $(a_1-n+1)/2-k$  & $\mathbb{R}_{>0}$ 
& $\mathbb{R}_{< 0}$  \\
\hline $2a_2+n+1$ & $(-n, -n+1)$ & $a_2 +1/2$  & $\mathbb{R}_{>0}$ 
& $\mathbb{R} $  \\
\hline $2a_2+n$ & $(-n+1, -n+2)$ &  $a_2 +1/2$   & $\mathbb{R}_{>0}$ 
& $\mathbb{R} $  \\
\hline $2a_2+n+1$ & $(-n, -n+1/2)$ & $a_2-k+1/2$   & $\mathbb{R}_{>0}$ 
& $\mathbb{R} $  \\
 \hline $-a_2-2n+j$   & $(-n+1, -n+2)$ & $-n-k+1/2$    & $ \mathbb{R}_{>0}$ 
 & $\mathbb{R} $  \\
\hline
\end{tabular}
\vspace{2mm}
\caption{Real zeros of $\pFq{3}{2}{-n,\  a_1,\ a_2}{ b_1,\ b_2}{x}$, as described in Proposition \ref{prop3F2Case3Newversion}. In each appearance in the table above, $k\in \mathbb Z_n $, $j\in \{0,1\}$, and $t\in \mathbb Z_n \cup  \rr_{>n-2}$. Recall that we always assume \eqref{assumptionsHGFNew}, so that $  a_i \notin (-\Z_n)$, and the polynomial is of degree exactly $n$. Moreover, a zero at $x=0$ appears only when $b_i \in (-\Z_n)$.
  }
    \label{tab:3F2case3}
\end{table}

As we have mentioned a few times, we do not claim that our approach is universal, and there are several results in the literature that we do not yet know how to prove using our method. To bridge this gap, we could try to use a combination of additive and multiplicative convolutions, similar to what was done in Proposition \ref{Prop:410new}. However, we have not been able to find new combinations of parameters that lead to real-rooted families using this idea.

Also, it is tempting to use the identity 
\begin{equation}
    \label{multIdentityCancellation3}
    \pFq{3}{2}{-n,\ a_1, \ a_2}{b_1,\ b_2}{x}=\pFq{3}{2}{-n,\ a_1, \ a_3}{b_1, \ b_2}{x} \boxtimes_n \pFq{2}{1}{-n, \ a_1 }{a_3 }{x},
\end{equation}
to extend previous assertions to some other combination of parameters. Nevertheless, at this stage, it does not lead to new results. The reason is that all real-rooted  ${}_3\mathcal{F}_2$ polynomials obtained so far are a result of the multiplicative convolution of more elementary ``blocks'', so that by replacing a parameter in a factor we get another factorization that was already considered.

However, we can use \eqref{multIdentityCancellation3} to extend some results in the literature to wider regions of parameters. We finish this section by illustrating it with some examples from \cite{Driver/Jordaan}.

For example, in \cite[Theorem 3.6]{Driver/Jordaan} it was proved  that
$$
\pFq{3}{2}{-n,\ n+1, \ 1/2}{b_1,\ 2-b_1}{x} \in \P_n(0,1)
$$
if $b_1\in (0,2)$. A multiplicative convolution with the polynomial
$$
\pFq{2}{1}{-n, \ a }{n+1 }{x}\in \P_n(\R_{>0})  \quad \text{for } a>2n-1,
$$
satisfying conditions from Row 2 of Table \ref{tab:2F1}, shows that for $a>2n-1$, 
$$
\pFq{3}{2}{-n,\ a, \ 1/2}{b_1,\ 2-b_1}{x} \in \P_n(\R_{>0}).
$$

Analogously, by \cite[Theorem 3.3]{Driver/Jordaan},
$$
\pFq{3}{2}{-n,2b_1+n, \ b_1-1/2}{b_1,\ 2b_1-1}{x} \in \P_n(\R_{>0})
$$
if $b_1>0$. Thus, a multiplicative convolution with the polynomial
$$
\pFq{2}{1}{-n, \ a_1 }{2b_1+n }{x}\in \begin{cases}
     \P_n(\R_{<0})  &\text{for }    a_1 < -n+1, \\
    \P_n(\R_{>0})  &\text{for }    a_1>2(b_1+n-1),
     \end{cases}
$$
satisfying conditions from Rows 1 and 2 of Table \ref{tab:2F1}, shows that for $ a>2(b_1+n-1)$, 
$$
\pFq{3}{2}{-n,a_1, \ b_1-1/2}{b_1,\ 2b_1-1}{x} \in 
\begin{cases}
     \P_n(\R_{<0})  &\text{for }    a_1 < -n+1, \, b_1>0,\\
    \P_n(\R_{>0})  &\text{for }    a_1 >2(b_1+n-1), \, b_1>0.
     \end{cases}
$$
 
Finally, by \cite[Theorem 3.4]{Driver/Jordaan},
$$
\pFq{3}{2}{-n,  2b_1+n-1, \ a_2}{b_1,\ b_2}{x} \in \P_n(\R_{>0}) 
$$
if $a_2=b_1-1/2$, $b_2=2b_1-2$, and $b_1>1$. 
Taking a multiplicative convolution with the polynomial
$$
\pFq{2}{1}{-n, \ a_1 }{2b_1+n-1 }{x}\in \begin{cases}
     \P_n(\R_{<0})  &\text{for }    a_1 < -n+1, \\
    \P_n(\R_{>0})  &\text{for }    a_1 >2(b_1+n)-3,
     \end{cases}
$$
satisfying conditions from Rows 1 and 2 of Table \ref{tab:2F1}, it shows that
$$
\pFq{3}{2}{-n,\ a_1, \ a_2}{b_1,\ b_2}{x} \in \begin{cases}
     \P_n(\R_{<0})  &\text{for }    a_1 < -n+1, \\
    \P_n(\R_{>0})  &\text{for }    a_1>2(b_1+n)-3,
     \end{cases}
$$
again, if $a_2=b_1-1/2$, $b_2=2b_1-2$, and $b_1>1$.

We summarize these final results in Table~\ref{tab:3F2case3extension}. 

\begin{table}[h]
    \centering
    \begin{tabular}{|c|c|c|c|c|}
    \hline $\mathbf{a}_1$ &  $\mathbf{a}_2$ & $\mathbf{b_1}$&$\mathbf{b_2}$ & \textbf{Roots in}  \\
     \hline
$\R_{>2n-1} $ & $1/2$  & $(0,2)$ & $2-b_1$ & $\R_{>0} $  \\
    \hline $  \R_{>2(b_1+n-1)}$  & $ b_1-1/2$ & $\rr_{>0}$   & $2b_1-1$ & $\R_{>0}  $  \\ 
 \hline $  \R_{<-n+1}$  & $ b_1-1/2$ & $\rr_{>0}$   & $2b_1-1$ & $\R_{<0}  $  \\ 
  \hline $  \R_{>2(b_1+n)-3}$ & $ b_1-1/2$ & $\rr_{>1}$   & $2b_1-2$ & $\R_{>0}  $  \\ 
 \hline $  \R_{<-n+1}$ & $ b_1-1/2$  & $\rr_{>1}$   & $2b_1-2$ & $\R_{<0}  $  \\ 
 \hline
\end{tabular}
\vspace{2mm}
\caption{Some additional cases of real zeros of $\pFq{3}{2}{-n,\  a_1,\ a_2}{ b_1,\ b_2}{x}$.  Recall that we always assume \eqref{assumptionsHGFNew}, so that $
  a_i \notin (-\Z_n)$, and the polynomial is of degree exactly $n$.  
  }
    \label{tab:3F2case3extension}
\end{table}

\section{Finite free probability and asymptotics} \label{sec:free.prob.asymototics} 

\subsection{Free probability}
\label{ssec:free} \ 

The goal of this section is to briefly explain how we can recast the previous results in the framework of free probability.
Free probability is a theory that studies non-commutative random variables, and it is especially useful in the study of the spectra of large random matrices. 

There are several natural parallels between the commutative and non-commutative theories. Since the concept of independence in classical probability theory is commutative in nature,  it is replaced by the notion of ``freeness'' or free independence, which is better suited to non-commutative random variables. Moreover, the central operations are the free additive $\boxplus$ and the free multiplicative convolution $\boxtimes$ of the measures which naturally correspond to the sum and multiplication of free random variables. The study of the free convolutions $\boxplus$ and $\boxtimes$ can be addressed either using the original Voiculescu's analytic tools such as $R$-transform and $S$-transform, or by the combinatorial theory developed by Nica and Speicher that makes use of free cumulants and noncrossing set partitions. Throughout this section, we assume that the reader has some familiarity with the theory of free probability; the standard references are \cite{DanKenAndu} for the analytical perspective and \cite{nica2006lectures} for the combinatorics perspective.

The connection between the convolutions of polynomials and free probability (reason for the name of "finite free" convolutions) was first noticed by Marcus, Spielman, and Srivastava in \cite{MR4408504}, when they used Voiculescu's $R$-transform and $S$-transform to improve the bounds on the largest root of a convolution of two real-rooted polynomials. This connection was explored further in \cite{marcus}, where Marcus defined a finite $R$-transform and $S$-transform that are related to Voiculescu's transforms in the limit. Using finite free cumulants, Arizmendi and Perales \cite{arizmendi2018cumulants}, showed that finite free additive convolution becomes a free additive convolution. This was later proved for the multiplicative convolution by Arizmendi, Garza-Vargas and Perales \cite{arizmendi2021finite}.

There is a natural way to associate a probability measure with a polynomial: given a polynomial $p$ of degree $n$ and roots $\lambda_j(p)$, $j=1, \dots, n$ (not necessarily all distinct),  its (normalized) \textbf{zero counting measure} (also known in this context as the \textbf{empirical root distribution} of $p$) is
\begin{equation}
    \label{def:zerocountingmeasure}
    \mu (p):=\frac{1}{n}\sum_{j=1}^n \delta_{\lambda_j(p)},
\end{equation}
where $\delta_z$ is the Dirac delta (unit mass) placed at the point $z$. The corresponding moments of $\mu(p)$ (which we also call the moments of $p$, stretching the terminology a bit) are
$$
m_k(p):=\frac{1}{n}\sum_{j=1}^n \lambda_j^k(p)=\int x^k \, d\mu_p, \quad k=0, 1, 2,\dots
$$

As mentioned above, the connection between finite and standard free probability is revealed in the asymptotic regime, when we let the degree $n\to \infty$. We say that the sequence of polynomials $ \mathfrak p= \left(p_n\right)_{n=1}^{\infty}$ such that each $p_n$ is real-rooted and of degree exactly $n$ \textbf{(weakly) converges} (or converges in moments) if there is a probability measure $\nu(\mathfrak p)$ on $\R$ with all its moments finite such that
$$
\lim_{n\to\infty} m_k(p_n)=m_k(\nu (\mathfrak p)), \qquad  k= 0, 1,2,\dots
$$
Note that if the moment problem for $\nu (\mathfrak p)$ is determined, this implies the weak-* convergence of the sequence $\mu (p_n)$ to $\nu (\mathfrak p)$.

\begin{proposition}[Corollary 5.5 in \cite{arizmendi2018cumulants}, and Theorem 1.4 in \cite{arizmendi2021finite}]
    \label{prop:finiteAsymptotics}
    Let $\mathfrak p:=\left(p_n\right)_{n=1}^{\infty}$ and $\mathfrak  q:=\left(q_n\right)_{n=1}^{\infty}$ be two sequences of real-rooted polynomials as above, and let $\nu(\mathfrak p)$ and $\nu(\mathfrak q)$ be two compactly supported probability Borel measures on $\R$ such that $\mathfrak p$ (respectively, $\mathfrak q$) weakly converges to $\nu(\mathfrak p)$ (respectively, $\nu(\mathfrak q)$). Then
    \begin{enumerate}
        \item[(i)]  $\left(p_n \boxplus_n q_n\right)_{n=1}^{\infty}$ weakly converges to $\nu(\mathfrak p) \boxplus \nu(\mathfrak q)$.
        \item[(ii)] if, additionally, for all sufficiently large $n$, $p_n, q_n\subset \pp_n(\rr_{>0})$ then $\left(p_n \boxtimes_n q_n\right)_{n=1}^{\infty}$ weakly converges to $\nu(\mathfrak p) \boxtimes \nu(\mathfrak q)$.
    \end{enumerate}
\end{proposition}

These results imply that, in the limit $n\to \infty$, we can replace the finite free convolution with the standard free convolution of measures. Thus, by combining this property with the results of this paper, we can systematically study the asymptotics of the root counting measures of families of hypergeometric polynomials. 

In the rest of this section, we illustrate these ideas in the simplest cases. A deeper analysis (in particular, with applications in approximation theory) is one of the goals of future work.

\subsection{Parameter rescaling}
\label{ssec:interpretation} \

In order to obtain nontrivial sequences of weakly converging hypergeometric polynomials, we need to allow the parameters $\bm a$ and $\bm b$ to depend on degree $n$. To simplify the presentation, we introduce the following notation:
\begin{notation}
Given $i,j,n\in \nn$, $\bm a =(a_1, \dots, a_i)\in \R^i$ and $\bm b =(b_1, \dots, b_j)\in \R^j$, we denote by $\HGP{n}{\bm b}{\bm a }(x)$ the unique monic polynomial of degree $n$ with coefficients in representation \eqref{monicP} given by
$$
e_k\left(\HGP{n}{\bm b}{\bm a}\right):=\binom{n}{k} \frac{\falling{\bm b n}{k}}{\falling{\bm a n}{k}}, \qquad \text{for }k=1,\dots,n.
$$
\end{notation}
In order to avoid indeterminacy, in this section we assume that
$$
a_s\not \in \left\{\tfrac{1}{n},\tfrac{2}{n}, \dots, \tfrac{n-1}{n}\right \}, \quad s=1, \dots, i.
$$

There is a direct connection of the polynomials we just introduced with the  hypergeometric polynomials in standard normalization:
\begin{equation}
\begin{split}
\HGP{n}{\bm b}{\bm a }(x) & =\frac{(-1)^n}{\falling{\bm a n}{n}} \pFq{i+1}{j}{-n, \bm a n -n+1}{\bm b n -n+1}{x} \\
& =\frac{(-1)^n\falling{\bm b n}{n}}{\falling{\bm a n}{n}} \HGF{i+1}{j}{-n, \bm a n -n+1}{\bm b n -n+1}{x} ,
\end{split}
\end{equation}
where $\bm c n -n+1$ means that we multiply each entry of $\bm c$ by $n$ and then  add $-n+1$. 

With the new notation, the simplest families of real rooted polynomials look as follows:

\begin{description}
    \item[Identity for the multiplicative convolution] $\HGP{n}{\bm a}{\bm a}(x)=\HGP{n}{-}{-}(x)=(x-1)^n$.

     \item[Identity for the additive convolution] $\HGP{n}{0}{\bm a}(x)=x^n$.

     \item[Laguerre polynomials] $\HGP{n}{b}{-}$,
     \begin{itemize}
         \item $\HGP{n}{b}{-}\in \P(\rr_{> 0})$ when $b>1-\frac{1}{n}$. 
        \item $\HGP{n}{b}{-}\in \P(\rr_{\geq 0})$ when $b\in\{\frac{1}{n},\frac{2}{n}, \dots, \frac{n-1}{n}\}$, with a multiplicity of $(1-b)n$ at $0$.
        \item $\HGP{n}{b}{-}\in \P(\rr)$ when $b\in (\frac{n-2}{n},\frac{n-1}{n})$.
     \end{itemize}

     \item[Bessel polynomials] $\HGP{n}{-}{a}$,
     \begin{itemize}
         \item $\HGP{n}{-}{a}\in \P(\rr_{< 0})$ when $a<0$.
        \item $\HGP{n}{-}{a}\in \P(\rr)$ when $a\in (0,\frac{1}{n})$.
     \end{itemize}

     \item[Jacobi polynomials] $\HGP{n}{b}{a}$,
     \begin{itemize}
         \item $\HGP{n}{b}{a}\in \P([0,1])$ when $b >1$ and $a>b+1$.
        \item $\HGP{n}{b}{a}\in \P(\rr_{<0})$ when $b>1$ and $a<0$.
        \item $\HGP{n}{b}{a}\in \P(\rr_{>0})$ when $a<0$ and $b<a-1$. 
     \end{itemize}
\end{description}
Take note that the case of Jacobi polynomials does not cover all the combination of parameters that lead to real-rooted polynomials; the reader is referred to Table \ref{tab:2F1} for further details. 

One combination that is particularly interesting corresponds to polynomials with only roots at 1 and 0:
$$
\HGP{n}{k/n}{1}(x) =(x-1)^{k}x^{n-k}, \qquad \text{for }k=0,1,2,\dots, n.
$$

\medskip

\begin{remark}
    In the realm of finite free probability, the Laguerre polynomials were first studied by Marcus \cite[Section 6.2.3]{marcus} using the finite $R$-transform and later in \cite{arizmendi2018cumulants} using finite free cumulants.  To our knowledge, the families of Bessel and Jacobi polynomials have not been studied in this context, except for some particular cases, such as Gegenbauer (or ultraspherical) polynomials that appeared in \cite[Section 6]{gribinski2022rectangular}. 
\end{remark}

\medskip

Notice that our previous results can be easily rewritten in the new notation. For instance, consider tuples $\bm a_1,\bm a_2,\bm a_3,\bm b_1,\bm b_2,\bm b_3 $ of sizes $i_1,i_2,i_3,j_1,j_2,j_3$, respectively, then two reciprocal polynomials from Remark \ref{rem:pFq.reversed} are of the form
\begin{equation}
\label{eq.reversed.freeprob}
\HGP{n}{\bm b}{\bm a }(x) \qquad\text{and}\qquad \HGP{n}{-\bm a+1-1/n}{-\bm b +1-1/n }\left((-1)^{i+j}x\right).  
\end{equation}

The multiplicative convolution (Theorem \ref{thm:multiplicativeConvHG}) works in exactly the same way:
\begin{equation}
\label{eq:multiplicative.rephrased}
\HGP{n}{\bm b_1}{\bm a_1 } \boxtimes_n \HGP{n}{\bm b_2}{\bm a_2 }=\HGP{n}{\bm b_1,\ \bm b_2 }{\bm a_1,\ \bm a_2}.
\end{equation}
And the additive convolution, specifically Corollary \ref{cor.additive.conv.as.poruduct.HG}, can be rephrased as follows:  assume that  the following factorization holds,
\begin{equation}
\label{eq:additive.rephrased.1}  
\HGF{j_1}{i_1}{-n\bm  b_1}{-n\bm a_1}{x}  \HGF{j_2}{i_2}{-n\bm b_2}{-n\bm a_2}{ x}=\HGF{j_3}{i_3}{-n\bm b_3}{-n\bm a_3}{x}.
\end{equation}
Then, considering the signs $s_l=(-1)^{i_l+j_l+1}$ for $l=1,2,3$, we have that
\begin{equation}
 \label{eq:additive.rephrased.2}   
\HGP{n}{\bm b_1}{\bm a_1 }(s_1 x) \boxplus_n \HGP{n}{\bm b_2}{\bm a_2 }(s_2 x) = \HGP{n}{\bm b_3}{\bm a_3 }(s_3 x).
\end{equation}

\medskip

The same applies to examples from Section \ref{sec:additive} that provide nontrivial cases with interesting interpretation from the point of view of finite free probability:
\begin{itemize}
\item From Example \ref{example37} it follows that 
$$
\HGP{n}{b_1}{-} \boxplus_n \HGP{n}{b_2}{-} = \HGP{n}{b_1+b_2}{-}, \qquad  b_1, b_2\in \rr.
$$

The Laguerre polynomials are one of the basic families of polynomials studied within the framework of finite free probability. This result is a direct consequence of the finite $R$-transform of the Laguerre polynomials calculated by Marcus \cite[Section 6.2.3]{marcus}. Equivalently, this result follows from the fact that all finite free cumulants of a Laguerre polynomial $\HGP{n}{b}{-}$ are all equal to $b$ \cite[Example 6.2]{arizmendi2018cumulants}. 

\item Example \ref{example38} we obtain
$$\HGP{n}{a_1+a_2-b}{-} \boxplus_n \HGP{n}{b-a_1,\ b-a_2}{b} = \HGP{n}{a_1,\ a_2}{b}, \qquad   a_1, a_2, b\in \rr.$$

\item Example \ref{example39} yields
$$\HGP{n}{a,\ b}{a+b-\tfrac{1}{2n}} ^{(\boxplus_n) 2} = \HGP{n}{2a,\ 2b,\ a+b}{a+b-\tfrac{1}{2n},\ 2a+2b}, \qquad  a, b\in \rr.$$

\item By Example \ref{exm:2F0.plus.2F0}, 
$$\HGP{n}{-}{2a} \boxplus_n \HGP{n}{-}{2b} = \text{Dil}_{4}\ \HGP{n}{a+b,\ a+b-\tfrac{1}{2n}}{2a,\ 2b,\ 2a+2b-\tfrac{1}{n}}, \qquad   a, b\in \rr,$$
where 
\begin{equation}
    \label{Defrescaling}
    \text{Dil}_{s} p(x):=s^n p(x/s), \quad s>0.
\end{equation}
\end{itemize}

In what follows, we derive some asymptotic formulas for the zero distribution of hypergeometric polynomials that can be represented in terms of some more elementary ``building blocks''. Thus, we start by discussing the zero distribution of the most basic sequences of real-rooted polynomials.

\subsection{Asymptotic results and new insights in free probability}\ 

Proposition \ref{prop:finiteAsymptotics} allows us to infer the asymptotic zero distribution of a sequence of polynomials that can be represented as a finite free convolution of simpler components, such as Laguerre, Bessel and Jacobi polynomials. We use the following notation for reparameterized polynomials:
\begin{align*}
\lag^{(a)}_n
&:= \text{Dil}_{1/n} \HGP{n}{a}{-} =\frac{1}{n^n} \HGP{n}{a}{-} (nx),\\
\bes^{(a)}_n(x)
&:= \text{Dil}_{n}  \HGP{n}{-}{a} =n^n \HGP{n}{-}{a} (x/n),\\
\jac^{(b,a)}_{n}(x)
&:= \HGP{n}{b}{a} (x), 
\end{align*}
where the dilation operator is defined in \eqref{Defrescaling}. 

Rescaling is needed in the case of Laguerre and Bessel polynomials, where otherwise the zeros would not be uniformly bounded (and weak compactness of the zero-counting measures is not guaranteed). With this definition, actually all three sequences, of Laguerre $\lag^{(b)}:=\left(\lag^{(b)}_n\right)_{n=1}^{\infty}$, Bessel $\bes^{(a)}:=\left(\bes^{(a)}_n\right)_{n=1}^{\infty}$ and Jacobi $\jac^{(b,a)}:=\left(\jac^{(b,a)}_n\right)_{n=1}^{\infty}$ polynomials, whenever real-rooted, are weakly converging. Their limiting measures are well known and can be computed using standard arguments from the theory of orthogonal polynomials:
\begin{itemize}
    \item Laguerre polynomials $\lag^{(b)}:=\left(\lag^{(b)}_n\right)_{n=1}^{\infty}$: for $b >1$, the limiting measure  is $\nu(\lag^{(b)})=\mu_{\text{MP}_b}$, the Marchenko-Pastur law with parameter $b$, which is an absolutely continuous probability measure on $[r_-, r_+]$, with
$$d\mu_{\text{MP}_b}= \frac{1}{2\pi} \frac{\sqrt{(r_+-x)(x-r_-)}}{x} dx, \qquad \text{where} \qquad r_\pm= b+1 \pm 2\sqrt{b}.$$
This distribution has been rediscovered many times and can be obtained from different perspectives: as the equilibrium measure on $\R_+$ in presence of an external field (see \cite[Ch.~IV]{MR1485778}), from the integral representation of the Laguerre polynomials \cite{MR1662685, MR1342385, MR1245742}, or from their differential equation \cite{MR1317299, MR1914357}. 

For $b\in (0,1)$, formulas \eqref{reductionLaguerre}--\eqref{HGFLaguerre} suggest that in the limit we get the Marchenko-Pastur distribution with an additional atom (mass point or Dirac delta) at $x=0$. This is indeed the case, as it can be easily established by either of the methods mentioned above. 

\item Bessel polynomials is $\bes^{(a)} :=\left(\bes^{(a)}_n\right)_{n=1}^{\infty}  $: for $a <0$,  the limiting measure $\mu_{\text{RMP}_a}$ is the reciprocal of a Marchenko-Pastur law of parameter $1-a$:
$$d\mu_{\text{RMP}_a}= \frac{-a}{2\pi} \frac{\sqrt{(r_+-x)(x-r_-)}}{x^2} dx, \qquad \text{where} \qquad r_\pm= \frac{1}{a-2 \pm 2\sqrt{1-a}},$$
which is a simple consequence of their connection with the Laguerre polynomials.

\item Jacobi polynomials $\jac^{(b,a)}  := \left(\jac^{(b,a)}_{n} \right)_{n=1}^\infty$: their asymptotic zero distribution $\mu_{b,a}:=\nu(\jac^{(b,a)})$ depends on which of the three major parameters regions  we are considering:
\begin{enumerate}[(J1)]
    \item When $b >1$ and $a>b+1$,  
\begin{equation*}
    d\mu_{b,a}=\frac{a}{4\pi}\frac{\sqrt{(r_+-x)(x-r_-)}}{x(1-x)}dx, \qquad \text{ where} \qquad  r_{\pm}=\left(\frac{\sqrt{a-b}\pm \sqrt{(a-1)b}}{a} \right)^2.
\end{equation*}

\item When $b >1$ and $a<0$,
\begin{equation*}
    d\mu_{b,a}=\frac{-a}{4\pi}\frac{x-1}{x}\sqrt{(r_+-x)(x-r_-)}dx, \qquad \text{ where} \qquad r_\pm=- \left(\frac{\sqrt{1-a}\mp \sqrt{b(b-a)}}{\sqrt{(1-a)(b-a)}\pm \sqrt{b}} \right)^2.
\end{equation*}

\item When $a<0$ and $b<a-1$. 
\begin{equation}
\label{eq.jacobi.limit.3}
    d\mu_{b,a}=\frac{-a x}{4\pi}\frac{\sqrt{(r_+-x)(x-r_-)}}{x-1}dx, \qquad \text{ where} \qquad r_\pm=\left(\frac{b-1}{\sqrt{(a-1)b}\mp \sqrt{a-b}} \right)^2.
\end{equation}
\end{enumerate}
\end{itemize}
As in the case of the Laguerre polynomials, these results follow considering either the weighed equilibrium problem for the logarithmic potential in an external field \cite[Ch.~IV]{MR1485778}), their integral representation \cite{MR1342385}, their differential equation \cite{MR1317299, MR1914357}, or even from their orthogonality relations \cite{MR2149265} combined with the Riemann-Hilbert method \cite{MR2124460}.

The distribution $\mu_{b,a}$ in case (J1) has already been studied in the realm of free probability in \cite[Definition 3.10]{yoshida2020remarks}.\footnote{\, We thank Katsunori Fujie and Yuki Ueda for mentioning to the third author this reference and the possible connection.} In that work, for $c,d>1$, the free beta distribution is given by $f\beta(c,d)=\mu_{c,c+d}$. This is hardly a surprise, since a direct consequence of equation \eqref{eq:multiplicative.rephrased} is the following identity:
$$\jac_n^{(c,c+d)}\boxtimes_n \lag_n^{(c+d)}= \lag_n^{(c)},
$$
which can be informally restated as that the Jacobi polynomials can be obtained as a quotient (in the free convolution sense) of two Laguerre polynomials.  By letting $n\to \infty$ we see that Yoshida's free beta distribution satisfies
$$f\beta(a,b) \boxtimes \text{MP}_{a+b}= \text{MP}_{a},$$
where as before, $\text{MP}_c$ denotes the Marchenko-Pastur distribution of parameter $c$. This is consistent with the fact that the free beta distribution can be obtained as a quotient of variables distributed according to a Marchenko-Pastur of different parameters.\footnote{\, We did not find an explicit formula but this is inferred implicitly in \cite{yoshida2020remarks} due to the relation between the free beta and the free beta prime distributions.}

In \cite{yoshida2020remarks}, the parameters are also allowed to be in the larger set $c,d>0$ (instead of $c,d>1$) with the additional condition that $c+d>1$, in which case, as the formulas \eqref{reductionJacobi1}--\eqref{reductionJacobi3new} for Jacobi polynomials suggest, the distribution can have atoms. Once again, this fact is rigorously established using any of the asymptotic methods mentioned above. 

After the previous discussion, one should be convinced that the Jacobi polynomials are the finite free analogue of the free beta distribution. This parallel can be observed also in random matrix theory. It is well known that the Hermite polynomials are tied to the study of the Gaussian Orthogonal Ensemble and the Laguerre polynomials are related to the real Wishart matrices (for a discussion of this in the realm of finite free probability see \cite[Section 5]{arizmendi2021finite}). In the same spirit, the Jacobi polynomials are related to the Jacobi ensembles, which are precisely those that can be constructed by taking the quotient of two Wishart ensembles. We refer the reader to \cite{collins2005product} for a detailed study of eigenvalues of Jacobi ensembles using Jacobi polynomials and free probability.  

\medskip

Our previous analysis combined with the results from Sections \ref{sec:convolutionHGP}--\ref{sec:compilation}, allows us to write the asymptotic zero distribution of diverse families of real-rooted hypergeometric polynomials in terms of free convolution of explicit distributions (Marchenko-Pastur, reciprocal Marchenko-Pastur, and Free Beta) enumerated above. We present some examples without going into explicit calculations, starting with the sequences of polynomials 
\begin{equation} \label{sequencerescaled}
    p:=\HGP{n}{ \bm b}{\bm a}, \quad \text{with } \bm a = \left( a_1,\dots,a_{i} \right) , \; \bm b = \left( b_1,\dots,b_{j} \right).
\end{equation} 

\begin{itemize}
    \item If $b_1,\dots,b_j>1$ and $a_1,\dots,a_{i}<0$, then the sequence \eqref{sequencerescaled} is weakly converging. By Theorem \ref{prop:42bis}, its asymptotic zero distribution can be written as
$$
\nu(p)=\mu_{\text{RMP}a_1}\boxtimes  \dots \boxtimes \mu_{\text{RMP}a_i}\boxtimes \mu_{\text{MP} b_1} \boxtimes \dots \boxtimes \mu_{\text{MP} b_j}.
$$

\item If $j\geq i$, $b_1,\dots,b_j>0$, and $a_1,\dots,a_{i}\in \rr$ such that $a_s\geq b_s+1$ for $s=1,\dots,i$, then the sequence \eqref{sequencerescaled} is weakly converging. By Theorem \ref{prop:42}, its asymptotic zero distribution can be written as
$$
\nu(p)=f\beta(b_1, a_1-b_1)\boxtimes  \dots \boxtimes f\beta(b_i, a_i-b_i)\boxtimes \mu_{\text{MP} b_{i+1}} \boxtimes  \dots \boxtimes \mu_{\text{MP} b_j}.
$$
\end{itemize}

Other examples of asymptotic distributions can be obtained from the multiplicative convolution discussed at the end of Section \ref{ssec:interpretation}:
\begin{itemize}
\item For $a_1, a_2, b\in \rr$ that satisfy the following conditions $a_1,a_2>1$, $b>a_1+1$, $b>a_2+1$, and $a_1+a_2-b>1$, we have the following relation between  real rooted Laguerre and Jacobi polynomials:
$$\HGP{n}{a_1+a_2-b}{-} \boxplus_n \left(\HGP{n}{b-a_1}{b} \boxtimes_n \HGP{n}{b-a_2}{-}\right) = (\HGP{n}{a_1}{b} \boxtimes_n \HGP{n}{a_2}{-}).$$

In the limit, this translates into a relation between the Marchenko-Pastur distribution and the free beta distribution:
$$\mu_{\text{MP}a_1+a_2-b} \boxplus (f\beta(b-a_1,a_1) \boxtimes \mu_{\text{MP}b-a_2}) = f\beta(a_1,b-a_1) \boxtimes \mu_{\text{MP}a_2}.$$

\item 
For $a, b\in \rr$ such that $a,b>1$, $a>b+1$, the following identity in terms of the real-rooted Laguerre and Jacobi polynomials holds
$$\left( \HGP{n}{a}{a+b-\tfrac{1}{2n}} \boxtimes_n \HGP{n}{b}{-}  \right) ^{(\boxplus_n) 2}  = \HGP{n}{2b}{a+b-\tfrac{1}{2n}} \boxtimes_n \HGP{n}{2a}{2a+2b} \boxtimes_n  \HGP{n}{a+b}{-}.$$

In the limit it becomes  
$$ 
(f\beta(a,b)\boxtimes \mu_{\text{MP}b} )^{\boxplus 2} = f\beta(2b,a-b) \boxtimes f\beta(2a,2b) \boxtimes \mu_{\text{MP}a+b}.
$$

\item For $a<0$ and $b<a-1$, we have the following relation between real-rooted Bessel and Jacobi polynomials:
$$\HGP{n}{-}{2a} \boxplus_n \HGP{n}{-}{2b} = \HGP{n}{a+b}{2a}\boxtimes_n \HGP{n}{a+b-\tfrac{1}{2n}}{2a+2b-\tfrac{1}{n}} \boxtimes_n \HGP{n}{-}{2b}.$$

Leaving $n\to \infty$ this yields the following identity:
$$\mu_{\text{RMP}a} \boxplus \mu_{\text{RMP}b} = \mu_{a+b,2a} \boxtimes \mu_{a+b, 2a+2b} \boxtimes \mu_{\text{RMP}2b},$$
where $\mu_{\text{RMP}c}$ stands for the reciprocal Marchenko-Pastur distribution of parameter $c$ and $\mu_{c,d}$ is the distribution obtained in Equation \eqref{eq.jacobi.limit.3}.
\end{itemize}

 \section*{Acknowledgments}

The first author was partially supported by Simons Foundation Collaboration Grants for Mathematicians (grant 710499).
He also acknowledges the support of the project PID2021-124472NB-I00, funded by MCIN/AEI/10.13039/501100011033 and by ``ERDF A way of making Europe'', as well as the support of of Junta de Andaluc\'{\i}a (research group FQM-229 and Instituto Interuniversitario Carlos I de F\'{\i}sica Te\'orica y Computacional). 

The third author was partially supported by the Simons Foundation via Michael Anshelevich's grant. He expresses his gratitude for the warm hospitality and stimulating atmosphere at Baylor University.  

This work has greatly benefited from our discussions with multiple colleagues: Octavio Arizmendi, Kathy Driver, Katsunori Fujie, Jorge Garza-Vargas, Kerstin Jordaan, Dmitry Karp, Vladimir Kostov, Ana Loureiro, Boris Shapiro,  Yuki Ueda, Wadim Zudilin, to mention a few (this list is incomplete). The discussions that originated this project started at a Brazos Analysis Seminar, and we are grateful to the organizers for giving us this opportunity.

%\bibliographystyle{abbrv}
%\bibliographystyle{alpha}
%\bibliography{References}

 \end{document}